\documentclass{amsart}
\usepackage{setspace}
\usepackage{a4}
\usepackage{amssymb,amsmath,amsthm,latexsym}
\usepackage{amsfonts}
\usepackage{amsfonts}
\usepackage{graphicx}
\usepackage{textcomp}
\usepackage{cite}
\newtheorem{theorem}{Theorem}[section]

\newtheorem{corollary}[theorem] {Corollary}
\newtheorem{definition}[theorem]{Definition}

\newtheorem{lemma} [theorem]{Lemma}

\newtheorem{proposition}[theorem]{Proposition}
\newtheorem{remark}[theorem]{Remark}
\title{This is the title}
\usepackage{amssymb}
\usepackage{amssymb}
\usepackage{amssymb}
\usepackage{amssymb}
\usepackage{amsmath}
\usepackage{tikz}
\usepackage{hyperref}
\usepackage{enumerate}
\usepackage{mathtools}
\usepackage{amsmath}
\usepackage{tikz}
\usepackage{hyphenat}
\begin{document}
\begin{center}
{\bf{MULTIPLIERS FOR  OPERATOR-VALUED  BESSEL SEQUENCES,  GENERALIZED HILBERT-SCHMIDT AND TRACE CLASSES}}\\
 K. Mahesh Krishna$^1$,   P. Sam Johnson$^1$, and R. N. Mohapatra$^2$ \\
 $^1$Department of Mathematical and Computational Sciences\\ 
 National Institute of Technology Karnataka (NITK), Surathkal,  Mangaluru 575 025, India  \\
 Emails : kmaheshak@gmail.com, 
 sam@nitk.edu.in \\
 
 $^2$Department of Mathematics\\
 University of Central Florida, Orlando, FL., 32816, USA \\
 Email : ram.mohapatra@ucf.edu

\end{center}
	
\hrule
\vspace{0.5cm}
\textbf{Abstract}: Let $\{\lambda_n\}_n \in \ell^\infty(\mathbb{N})$. In 1960,  R. Schatten \cite{SCHATTEN} studied operators of the form $\sum_{n=1}^{\infty}\lambda_n (x_n\otimes \overline{y_n})$, where $\{x_n\}_n$, $\{y_n\}_n$ are orthonormal sequences in a Hilbert space.    In 2007,  P. Balazs \cite{BALAZS3} generalized some of the the results of Schatten \cite{SCHATTEN}. In this paper, we further generalize the result of Balazs by  studying the   operators of the form  $\sum_{n=1}^{\infty}\lambda_n (A^*_nx_n\otimes \overline{B^*_ny_n})$, where $\{A_n\}_n$     and $\{B_n\}_n$ are operator-valued Bessel sequences  and  $\{x_n\}_n$, $\{y_n\}_n$ are sequences in the Hilbert space such that $\{\|x_n\|\|y_n\|\}_n  \in \ell^\infty(\mathbb{N})$. We   generalize the classes of Hilbert-Schmidt  and trace class operators. 

	

\section{Introduction}

 Let $\mathcal{H},\mathcal{H}_0$ (resp. $\mathcal{X},\mathcal{Y}$) be  Hilbert spaces (resp.  Banach spaces) and  $\mathcal{B}(\mathcal{X},\mathcal{Y})$  (resp. $  \mathcal{K}(\mathcal{X},\mathcal{Y})$) be the Banach space of all bounded linear operators (resp. compact operators) from $\mathcal{X}$ to $\mathcal{Y}$. We set $\mathcal{B}(\mathcal{X})\coloneqq\mathcal{B}(\mathcal{X},\mathcal{X})$, $\mathcal{K}(\mathcal{X})\coloneqq\mathcal{K}(\mathcal{X},\mathcal{X})$ and  for $A \in \mathcal{B}(\mathcal{H})$, $ [A]\coloneqq (A^*A)^{1/2}$. 
 \begin{definition}\cite{SCHATTEN}
  For $x,y \in \mathcal{H}$, define	$x\otimes \overline{y}:\mathcal{H}\ni h \mapsto \langle h, y\rangle x\in \mathcal{H}$.
  \end{definition}
 In Chapter 1 of \cite{SCHATTEN}, Schatten made a detailed study of operators of the form
 \begin{align}\label{FIRST EQUATION}
 \sum_{n=1}^{\infty}\lambda_n (x_n\otimes \overline{y_n}),
 \end{align}
   where $\{\lambda_n\}_n$ is in $\ell^\infty(\mathbb{N})$ and  $\{x_n\}_n$, $\{y_n\}_n$ are orthonormal sequences in a Hilbert space. It is then showed that every  compact operator (what Schatten called as a completely continuous operator) is of  the form in (\ref{FIRST EQUATION}) with $\lambda_n\geq 0, \forall n \in \mathbb{N}$ and $\lambda_n\to 0$ as $n\to \infty$ (the spectral theorem for compact operators). Chapters 2 and 3 of \cite{SCHATTEN} contain Hilbert Schmidt $\mathcal{S}(\mathcal{H})$ and trace class operators $\mathcal{T}(\mathcal{H})$, respectively. These two chapters mainly contain results obtained by Schatten and von Neumann in their joint paper \cite{SCHATTENVONNEUMANN}. It was showed that both $\mathcal{S}(\mathcal{H})$ and $\mathcal{T}(\mathcal{H})$ admit norms under which they are complete and they are two sided star closed ideals in $\mathcal{B}(\mathcal{H})$. It is further proved  that the norm on $\mathcal{S}(\mathcal{H})$ comes from an inner product and hence  $\mathcal{S}(\mathcal{H})$ is a Hilbert space. Since a Hilbert space is self-dual, a natural  question occurs in Chapter 3 of \cite{SCHATTEN} is $-$  what is the dual of $\mathcal{T}(\mathcal{H})$? This is answered in Chapter 4, of \cite{SCHATTEN} which shows that $\mathcal{T}(\mathcal{H})^*=\mathcal{B}(\mathcal{H})$ and $\mathcal{K}(\mathcal{H})^*=\mathcal{T}(\mathcal{H})$. Chapter 5 of \cite{SCHATTEN} introduces the definitions of cross norms and norm ideals,  (studied in \cite{SCHATTEN4, SCHATTEN5, SCHATTEN6, SCHATTENVONNEUMANN1, SCHATTENVONNEUMANN}). Symmetric gauge functions are also introduced in Chapter 5 of \cite{SCHATTEN}.

 Consequently, Chapters 2, 3 and 5 of \cite{SCHATTEN} are extended to Schatten p-classes  \cite{RINGROSE, DIESTEL} and  operator ideals \cite{DEFANTFLORET, PIETSCH}. It was in 2007, when P. Balazs \cite{BALAZS3} generalized the operator in  (\ref{FIRST EQUATION}) by relaxing orthonormal sequences  $\{x_n\}_n$, $\{y_n\}_n$ to Bessel sequences (we refer \cite{CHRISTENSEN} for Bessel sequence and its properties). The first question while considering the  operator in (\ref{FIRST EQUATION}) is its existence. Whenever $\{x_n\}_n$ and  $\{y_n\}_n$ are orthonormal, Schatten considered (for $n, m \in \mathbb{N}$ and $h \in \mathcal{H}$), 
 \begin{align*}
 \left\|\sum_{k=n}^{m}\lambda_k (x_k\otimes \overline{y_k})h\right\|^2=\sum_{k=n}^{m}\sum_{r=n}^{m}\lambda_k\overline{\lambda_r}\langle h,  y_k\rangle\langle  y_r,h\rangle\langle x_k,  x_r\rangle= \sum_{k=n}^{m}|\lambda_k|^2|\langle h,  y_k\rangle|^2.
 \end{align*}
 Bessel's inequality now tells that the  operator in (\ref{FIRST EQUATION}) exists. This idea won't work when we drop the orthonormality. In the case $\{x_n\}_n$ and  $\{y_n\}_n$ are Bessel sequences, Balazs realized that the operator in (\ref{FIRST EQUATION}) exists and appears as the composition of three bounded linear operators (Theorem 6.1 in \cite{BALAZS3}) acting on Hilbert spaces whose existence comes from frame theory. Balazs and Stoeva studied invertibility of these operators in \cite{STOEVABALAZS1, STOEVABALAZS2, STOEVABALAZS3, STOEVABALAZS4, STOEVABALAZS5}.

 In Section \ref{MULTIPLIERSSECTION}  we generalize the work of Balazs by considering  operator-valued Bessel sequence  (see (\ref{GENERALOPERATOR})) and we derive various properties of it and continuity of multipliers. Section  \ref{HSCLASSSECTION} extends the class of Hilbert-Schmidt operators and Section \ref{TRACECLASSSECTION} extends  trace class operators. The classes which we define in Sections \ref{HSCLASSSECTION} and \ref{TRACECLASSSECTION} will depend upon a conjugate-linear isometry, an operator-valued orthonormal basis and a sequence in a Hilbert space (unlike Hilbert-Schmidt classes and trace classes which do not depend upon orthonormal bases of the Hilbert space (Lemma 2.1 and Lemma 2.3 in \cite{SCHATTENVONNEUMANN})). We show that generalized Hilbert-Schmidt class admits a semi-norm and  it is a two sided  star closed ideal in $\mathcal{B}(\mathcal{H})$. We are unable to do this for generalized trace class where we show this class is closed under multiplication and taking adjoints. In both Sections \ref{HSCLASSSECTION} and \ref{TRACECLASSSECTION} we follow the same strategy done in Chapters 2 and 3 of \cite{SCHATTEN}, respectively. 
  
  We remark here that certain generalizations of the operator studied by Balazs (in \cite{BALAZS3}) has been studied in \cite{FAROUGHIELNAZ, RAHIMI2, RAHIMIBALAZS2,  JAVANSHIRICHOUBIN, ARIASPACHECO}.

 In the remaining part of introduction we list some definitions and results which we use in the sequel. 
  
\begin{definition} \cite{DUFFINSCHAEFFER}
A sequence $\{x_n\}_n$ in $\mathcal{H}$ is said to be a frame if there exist $a,b>0$ such that
\begin{align*}
a\|h\|^2\leq \sum_{n=1}^{\infty} |\langle h, x_n\rangle|^2\leq b \|h\|^2, ~ \forall h \in \mathcal{H}.
\end{align*}
Constants $a$ and $b$ are called  frame bounds. If $a$ is allowed to take the value 0, then $\{x_n\}_n$ is called as a Bessel sequence with bound $b$.
\end{definition}
  
 \begin{definition}  \cite{SUN}
A collection  $\{A_n\}_n$ in  $\mathcal{B}(\mathcal{H},\mathcal{H}_0)$ is said to be an operator-valued Bessel sequence (g-Bessel sequence) with bound $b>0$	if 
$\sum_{n=1}^{\infty}\|A_nh\|^2\leq b\|h\|^2,\forall h \in \mathcal{H}.$
 \end{definition}
 It can be seen easily that if $\{A_n\}_n$ is   an operator-valued Bessel sequence   in  $\mathcal{B}(\mathcal{H},\mathcal{H}_0)$ with bound $b$, then 	$\|A_n\|\leq \sqrt{b}, \forall n \in \mathbb{N}$. In fact, $\|A_nh\|^2\leq \sum_{k=1}^{\infty}\|A_kh\|^2\leq b\|h\|^2, \forall h \in \mathcal{H}, \forall n \in \mathbb{N}$.
 \begin{definition}\cite{SUN}
 A collection  $\{A_n\}_n$ in  $\mathcal{B}(\mathcal{H},\mathcal{H}_0)$ is said to be  
 \begin{enumerate}[\upshape(i)]
 \item an operator-valued Riesz basis (g-Riesz basis)	if  $\{h \in \mathcal{H}: A_nh=0, \forall n \in \mathbb{N}\}=\{0\}$ and there exist $a,b>0$ such that for every finite  $\mathbb{S} \subseteq \mathbb{N}$,
 \begin{align*}
 a \sum_{n \in \mathbb{S}} \|y_n\|^2\leq \left\| \sum_{n \in \mathbb{S}} A_n^*y_n\right\|^2\leq b \sum_{n \in \mathbb{S}} \|y_n\|^2, ~\forall y_n \in \mathcal{H}_0.
 \end{align*}
 \item  an operator-valued orthonormal basis (g-basis)	if $\langle A_n^*y , A_m^*z\rangle =\delta_{n,m} \langle y,z\rangle,  \forall n, m \in \mathbb{N}  , \forall y,z \in \mathcal{H}_0$,  and $\sum_{n=1}^{\infty}\|A_nh\|^2=\|h\|^2,\forall h \in \mathcal{H}.$
  \end{enumerate}
 \end{definition}
 We refer the reader \cite{SUN}  for examples and properties of operator-valued orthonormal bases and Riesz bases. 
 \begin{definition}\cite{MAHESHKRISHNASAMJOHNSON}
A collection  $\{A_n\}_n$ in  $\mathcal{B}(\mathcal{H},\mathcal{H}_0)$ is said to be    an operator-valued 
\begin{enumerate}[\upshape(i)]
\item orthogonal sequence 	if $\langle A_n^*y , A_m^*z\rangle =\delta_{n,m} \langle y,z\rangle,  \forall n, m \in \mathbb{N}  , \forall y,z \in \mathcal{H}_0.$
\item orthonormal sequence 	if it is orthogonal and  $\sum_{n=1}^{\infty}\|A_nh\|^2\leq\|h\|^2,\forall h \in \mathcal{H}.$
\end{enumerate}	
 \end{definition}
We provide the proofs of following two theorems for the sake of reader.
 \begin{theorem}\cite{MAHESHKRISHNASAMJOHNSON} \label{ORTHONORMALBASISCRITERION}
 If  $\{A_n\}_n$ and  $\{B_n\}_n$ are two  operator-valued orthonormal bases in $\mathcal{B}(\mathcal{H},\mathcal{H}_0)$,  then there exists a unique unitary  $U \in \mathcal{B}(\mathcal{H})$ such that $A_n=B_nU, \forall n \in \mathbb{N}$.	
 \end{theorem}
 \begin{proof}
 (Existence) Define	$ U\coloneqq \sum_{n=1}^\infty B_n^*A_n.$ This operator exists in the strong-operator topology, since for every  $n,m \in \mathbb{N}$ with $n<m$ and $ h \in \mathcal{H},$ 
 	\begin{align*}
 	\left\|\sum_{j=n}^mB_j^*A_jh\right\|^2=\left\langle\sum_{j=n}^mB_j^*A_jh, \sum_{k=n}^mB_k^*A_kh\right\rangle= \sum_{j=n}^m\left\langle A_jh, B_j\left(\sum_{k=n}^mB_k^*A_kh\right) \right\rangle=\sum_{j=1}^n\|A_jh\|^2.
 	\end{align*}
 	 Now $ B_nU=B_n(\sum_{m=1}^\infty B_m^*A_m)=A_n,  \forall n \in \mathbb{N}.$ We now show that   $ U$ is unitary. For, 
 	 \begin{align*}
 	 UU^*=(\sum_{n=1}^\infty B_n^*A_n)(\sum_{m=1}^\infty A^*_mB_m)=\sum_{n=1}^\infty B_n^*(\sum_{m=1}^\infty A_nA^*_mF_m)=\sum_{n=1}^\infty B_n^*B_n=I_\mathcal{H} 
 	 \end{align*}
 	 and 
 	 \begin{align*}
 	  U^*U=(\sum_{n=1}^\infty A_n^*B_n)(\sum_{m=1}B^*_mA_m)=\sum_{n=1}^\infty A_n^*(\sum_{m=1}^\infty B_nB^*_mA_m)=\sum_{n=1}^\infty A_n^*A_n=I_\mathcal{H}.
 	 \end{align*}
 (Uniqueness) Let $W \in \mathcal{B}(\mathcal{H})$ also satisfies $B_nU=B_nW=A_n, \forall n \in \mathbb{N}$. Then $U=I_\mathcal{H}U=\sum_{n=1}^\infty B_n^*(B_nU)=\sum_{n=1}^\infty B_n^*(B_nW)=W.$
\end{proof}
 
 \begin{theorem}\cite{MAHESHKRISHNASAMJOHNSON}\label{RIESZBASISCRITERION}
If $\{F_n\}_n$ in $\mathcal{B}(\mathcal{H},\mathcal{H}_0)$ is an operator-valued orthonormal basis in $\mathcal{B}(\mathcal{H},\mathcal{H}_0)$ and $\{A_n\}_n$ is an operator-valued Riesz basis  in $\mathcal{B}(\mathcal{H},\mathcal{H}_0)$, then there exists a unique   invertible $T \in \mathcal{B}(\mathcal{H})$ such that $A_n=F_nT, \forall n \in \mathbb{N}$.
 \end{theorem}
 \begin{proof}
 (Existence) From the definition of operator-valued Riesz basis, there exists an operator-valued orthonormal basis $\{G_n\}_n$   in  $ \mathcal{B}(\mathcal{H},\mathcal{H}_0)$ and invertible $ R :\mathcal{H}\rightarrow \mathcal{H} $ such that   $ A_n=G_nR,  \forall n \in \mathbb{N} $. Define $ T\coloneqq\sum_{n=1}^\infty F_n^*G_nR.$ Since $\{F_n\}_n$ and  $\{G_n\}_n$ are orthonormal bases, similar to the  proof of Theorem \ref{ORTHONORMALBASISCRITERION}, $ T$ is well-defined. Now $ F_nT=G_nR=A_n,  \forall n \in \mathbb{N}$, 
 \begin{align*}
 T (R^{-1}(\sum_{k=1}^\infty G_k^*F_k))= (\sum_{n=1}^\infty F_n^*G_nR) (R^{-1}(\sum_{k=1}^\infty G_k^*F_k)) =\sum_{n=1}^\infty F_n^* (\sum_{k=1}^\infty G_nG_k^*F_k)=\sum_{n=1}^\infty F_n^*F_n =I_{\mathcal{H}}
 \end{align*}
 and 
 \begin{align*}
 (R^{-1}(\sum_{k=1}^\infty G_k^*F_k))T  =R^{-1}(\sum_{n=1}^\infty G_n^*F_n)(\sum_{k=1}^\infty F_k^*G_kR)=R^{-1}(\sum_{n=1}^\infty G_n^*(\sum_{k=1}^\infty F_nF_k^*G_kR))=R^{-1}(\sum_{n=1}^\infty G_n^*G_n)R =I_{\mathcal{H}}.
 \end{align*}
 (Uniqueness) Let $W \in \mathcal{B}(\mathcal{H})$ also satisfies $F_nT=F_nW=A_n, \forall n \in \mathbb{N}$, then $T=I_\mathcal{H}T=\sum_{n=1}^\infty F_n^*(F_nT)=\sum_{n=1}^\infty F_n^*(F_nW)=W.$	
 \end{proof}
  \begin{theorem}\cite{SCHATTEN}\label{POLARDECOMPOSITIONSCHATTEN} (Polar decomposition) 
  Let $A \in \mathcal{B}(\mathcal{H})$. Then there exists a partial isometry $W$ whose initial space is $\overline{[A](\mathcal{H})}$ and the final space is $\overline{A(\mathcal{H})}$, satisfying the following conditions. 
 \begin{enumerate}[\upshape(i)]
 \item $A=W[A]$.
  \item $[A]=W^*A$.
  \item $A^*=W^*[A^*]$.
  \item  $[A^*]=W[A]W^*$.
  \end{enumerate} 
  The above decomposition of $A$ is unique in the following sense: If $A=W_1B_1$ where $B_1\geq0$ and $W_1$ is a partial isometry with initial space $\overline{B_1(\mathcal{H})}$, then $B_1=[A]$ and $W_1=W$. Further, if $A$ is a finite rank operator, then we can take $W$ as unitary.
  \end{theorem}
  \begin{definition}\cite{SCHATTENVONNEUMANN}\label{SCHMIDTCLASS}
  Let $\{e_n\}_n$ be an orthonormal basis for $\mathcal{H}$. The Hilbert-Schmidt class is defined as 
  \begin{align*}
  \mathcal{S}(\mathcal{H})\coloneqq \left\{ A \in  \mathcal{B}(\mathcal{H}):\sum_{n=1}^{\infty}\|Ae_n\|^2<\infty \right\}
  \end{align*}
  with the norm of $A \in \mathcal{S}(\mathcal{H})$ is $\sigma(A)\coloneqq\left(\sum_{n=1}^{\infty}\|Ae_n\|^2\right)^{1/2}$.
  \end{definition}\label{TRACECLASS}
  \begin{definition}\cite{SCHATTENVONNEUMANN}
   Let $\{e_n\}_n$ be an orthonormal basis for $\mathcal{H}$. The trace class is defined as 
   \begin{align*}
   \mathcal{T}(\mathcal{H})\coloneqq \{AB:A,B \in \mathcal{S}(\mathcal{H})\}
   \end{align*}
   with the trace of $C \in \mathcal{T}(\mathcal{H})$ is  $\operatorname{Tr}(C)\coloneqq \sum_{n=1}^{\infty}\langle Ce_n,e_n \rangle $ and the norm of $C \in \mathcal{T}(\mathcal{H})$ is $\tau(C)\coloneqq\operatorname{Tr}([C])$.	
  \end{definition}
 \begin{definition}\cite{PIETSCH}
 An operator $T \in \mathcal{B}(\mathcal{X},\mathcal{Y})$ is called nuclear if there exist sequences  $\{f_n\}_n$ in $\mathcal{X}^*$ and $\{y_n\}_n$ in $\mathcal{Y}$ such that $Tx=\sum_{n=1}^{\infty}f_n(x)y_n, \forall x \in \mathcal{X}$. In this case, we define the nuclear-norm of $T$ as 
 \begin{align*}
 \|T\|_{\operatorname{Nuc}}\coloneqq \inf\left\{\sum_{n=1}^{\infty}\|f_n\|\|y_n\|:T \in  \mathcal{B}(\mathcal{X},\mathcal{Y}) \text{ is nuclear with } T=\sum_{n=1}^{\infty}f_n(\cdot)y_n\right\}.
 \end{align*}
 \end{definition}

 \section{Multipliers for operator-valued  Bessel sequences} \label{MULTIPLIERSSECTION}
 \begin{theorem}\label{DEFINITIONEXISTENCE}
 Let $\{A_n\}_n$ and $\{B_n\}_n$ be  operator-valued Bessel sequences in   $\mathcal{B}(\mathcal{H},\mathcal{H}_0)$ with bounds $b$, $d$,  respectively.  If $\{\lambda_n\}_n \in \ell^\infty(\mathbb{N})$, and  $\{x_n\}_n$, $\{y_n\}_n$ are sequences in $\mathcal{H}_0$ such that $\{\|x_n\|\|y_n\|\}_n $ $ \in \ell^\infty(\mathbb{N})$, then the map
 
 \begin{align*}
 T: \mathcal{H} \ni h \mapsto \sum_{n=1}^{\infty}\lambda_n (A^*_nx_n\otimes \overline{B^*_ny_n})h \in \mathcal{H}
 \end{align*}
 is a well-defined bounded linear operator with norm at most $\sqrt{bd}\|\{\lambda_n\}_n\|_\infty\sup_{n\in \mathbb{N}}\|x_n\|\|y_n\|.$
 \end{theorem} 
 \begin{proof}
 Let $n,m \in \mathbb{N}$ with $n\leq m$. Then for each $h \in \mathcal{H}$,
 
 \begin{align*}
 &\left\|\sum_{k=n}^{m}\lambda_k (A^*_kx_k\otimes \overline{B^*_ky_k})h\right\|=\sup_{g\in \mathcal{H},\|g\|\leq 1}\left|\left \langle \sum_{k=n}^{m}\lambda_k (A^*_kx_k\otimes \overline{B^*_ky_k})h, g\right \rangle \right| \\
 &=\sup_{g\in \mathcal{H},\|g\|\leq 1}\left|\sum_{k=n}^{m}\lambda_k \langle h, B^*_ky_k\rangle \langle A^*_kx_k, g\rangle \right|
 \leq \sup_{g\in \mathcal{H},\|g\|\leq 1}\sum_{k=n}^{m}|\lambda_k \langle h, B^*_ky_k\rangle \langle A^*_kx_k, g\rangle |\\
 &=\sup_{g\in \mathcal{H},\|g\|\leq 1}\sum_{k=n}^{m}|\lambda_k \langle  B_kh, y_k\rangle \langle x_k, A_kg\rangle |
 \leq \sup_{g\in \mathcal{H},\|g\|\leq 1}\sum_{k=n}^{m}|\lambda_k |\|B_kh\|\|y_k\|\| x_k\|\|A_kg\|\\
 &\leq \sup_{n\in \mathbb{N}}|
 \lambda_n|\sup_{n\in \mathbb{N}}\|x_n\|\|y_n\|\sup_{g\in \mathcal{H},\|g\|\leq 1}\sum_{k=n}^{m} \|B_kh\|\|A_kg\|\\
 &\leq \sup_{n\in \mathbb{N}}|
 \lambda_n|\sup_{n\in \mathbb{N}}\|x_n\|\|y_n\|\sup_{g\in \mathcal{H},\|g\|\leq 1} \left(\sum_{k=n}^{m} \|B_kh\|^2\right)^\frac{1}{2}\left(\sum_{k=n}^{m} \|A_kg\|^2\right)^\frac{1}{2}\\
 &\leq \sup_{n\in \mathbb{N}}|
 \lambda_n|\sup_{n\in \mathbb{N}}\|x_n\|\|y_n\|\sup_{g\in \mathcal{H},\|g\|\leq 1} \left(\sum_{k=n}^{m} \|B_kh\|^2\right)^\frac{1}{2}\left(\sum_{k=n}^{\infty} \|A_kg\|^2\right)^\frac{1}{2}\\
 &\leq \sqrt{b}\sup_{n\in \mathbb{N}}|
 \lambda_n|\sup_{n\in \mathbb{N}}\|x_n\|\|y_n\|\left(\sum_{k=n}^{m} \|B_kh\|^2\right)^\frac{1}{2}\sup_{g\in \mathcal{H},\|g\|\leq 1}\|g\|\\
 &=\sqrt{b}\sup_{n\in \mathbb{N}}|
 \lambda_n|\sup_{n\in \mathbb{N}}\|x_n\|\|y_n\|\left(\sum_{k=n}^{m} \|B_kh\|^2\right)^\frac{1}{2}, 
 \end{align*}
 and  $\sum_{k=1}^{\infty}\|B_kh\|^2$ converges with $\sum_{k=1}^{\infty}\|B_kh\|^2\leq d \|h\|^2$. Hence $T$ is well-defined linear. Above calculations also show that  $\|T\|\leq \sqrt{ab}\sup_{n\in \mathbb{N}}|\lambda_n|\sup_{n\in \mathbb{N}}\|x_n\|\|y_n\|.$
 \end{proof}
 \begin{corollary}
  Let $\{A_n\}_n$ be an operator-valued orthonormal sequence in  $\mathcal{B}(\mathcal{H},\mathcal{H}_0)$, $\{B_n\}_n$ be an operator-valued Bessel sequence in  $\mathcal{B}(\mathcal{H},\mathcal{H}_0)$ with bound $b$.  If $\{\lambda_n\}_n \in \ell^\infty(\mathbb{N})$, $\{x_n\}_n$, $\{y_n\}_n$ are sequences in $\mathcal{H}_0$ such that $\{\|x_n\|\|y_n\|\}_n \in \ell^\infty(\mathbb{N})$, then the map $T: \mathcal{H} \ni h \mapsto \sum_{n=1}^{\infty}\lambda_n (A^*_nx_n\otimes \overline{B^*_ny_n})h \in \mathcal{H}$  is a well-defined bounded linear operator with norm at most $\sqrt{b}|\|\{\lambda_n\}_n\|_\infty\sup_{n\in \mathbb{N}}\|x_n\|\|y_n\|.$
 \end{corollary}
 \begin{proof}
 Even though this is a Corollary of Theorem  \ref{DEFINITIONEXISTENCE}, we shall write a direct argument using orthonormality of $A_n$'s. 
  Let $n,m \in \mathbb{N}$ with $n\leq m$ and $h \in \mathcal{H}$. Consider
  
 \begin{align*}
 &\left\|\sum_{k=n}^{m}\lambda_k (A^*_kx_k\otimes \overline{B^*_ky_k})h\right\|^2=\left\langle \sum_{k=n}^{m}\lambda_k (A^*_kx_k\otimes \overline{B^*_ky_k})h,\sum_{r=n}^{m}\lambda_r (A^*_rx_r\otimes \overline{B^*_ry_r})h \right \rangle \\
 &=\left\langle\sum_{k=n}^{m}\lambda_k\langle h, B^*_ky_k\rangle A^*_kx_k , \sum_{r=n}^{m}\lambda_r\langle h, B^*_ry_r\rangle A^*_rx_r \right \rangle
 = \sum_{k=n}^{m}\lambda_k\langle h, B^*_ky_k\rangle\sum_{r=n}^{m}\overline{\lambda_r}\langle  B^*_ry_r,h\rangle \langle A_k^*x_k, A_r^*x_r\rangle \\
 &= \sum_{k=n}^{m}\lambda_k\langle h, B^*_ky_k\rangle\overline{\lambda_k}\langle  B^*_ky_k,h\rangle\langle x_k, x_k\rangle
 =\sum_{k=n}^{m}|\lambda_k|^2|\langle h, B^*_ky_k\rangle|^2\|x_k\|^2 \\
 &\leq \sup_{n\in \mathbb{N}}|\lambda_n|^2\sum_{k=n}^{m}|\langle h, B^*_ky_k\rangle|^2\|x_k\|^2
 =\sup_{n\in \mathbb{N}}|\lambda_n|^2\sum_{k=n}^{m}|\langle  B_kh,y_k\rangle|^2\|x_k\|^2\\
 &\leq \sup_{n\in \mathbb{N}}|\lambda_n|^2\sum_{k=n}^{m}\|  B_kh\|^2\|y_k\|^2\|x_k\|^2\leq  \sup_{n\in \mathbb{N}}|\lambda_n|^2\sup_{n\in \mathbb{N}}\|x_n\|^2\|y_n\|^2\sum_{k=n}^{m}\|  B_kh\|^2,
 \end{align*}
the last sum converges.
 \end{proof} 
  \begin{corollary}\label{COROLLARYFIRST}
 Theorem \ref{DEFINITIONEXISTENCE}  holds by replacing the condition  $\{\|x_n\|\|y_n\|\}_n \in \ell^\infty(\mathbb{N})$ with the condition $\{\|x_n\|\}_n, \{\|y_n\|\}_n \in \ell^\infty(\mathbb{N})$.
  \end{corollary}
  \begin{remark}
  Corollary \ref{COROLLARYFIRST} can also be derived by using Theorem 6.1 in \cite{BALAZS3}. In fact, if $\{\|x_n\|\}_n, \{\|y_n\|\}_n $ $ \in \ell^\infty(\mathbb{N})$, then we observe that both $\{A_n^*x_n\}_n$, $ \{B_n^*y_n\}_n$ are Bessel sequences. For, $\sum_{n=1}^{\infty}|\langle h, A_n^*x_n\rangle |^2=\sum_{n=1}^{\infty}|\langle A_nh, x_n\rangle |^2\leq \sum_{n=1}^{\infty}\|A_nh\|^2\|x_n\|^2\leq \sup_{n\in \mathbb{N}}\|x_n\|^2\sum_{n=1}^{\infty}\|A_nh\|^2\leq a\sup_{n\in \mathbb{N}}\|x_n\|^2\|h\|^2$. Similarly  $\sum_{n=1}^{\infty}|\langle h, B_n^*y_n\rangle |^2\leq b\sup_{n\in \mathbb{N}}\|y_n\|^2\|h\|^2, \forall h \in \mathcal{H}$. Now \text{\upshape(i)} in Theorem 6.1 in \cite{BALAZS3} says that $T$ is a well-defined bounded linear operator with  $\|T\|\leq \sqrt{ab}\|\{\lambda_n\}_n\|_\infty\sup_{n\in \mathbb{N}}\|x_n\|\sup_{n\in \mathbb{N}}\|y_n\|.$ 
 \end{remark}
  A partial converse of Theorem \ref{DEFINITIONEXISTENCE} is given in Theorem \ref{CHARACTERIZATIONRESULTORTHO} (which extends Theorem 1 of Chapter 1 in \cite{SCHATTEN}).
 \begin{theorem}\label{CHARACTERIZATIONRESULTORTHO}
  Let $\{A_n\}_n$, $\{B_n\}_n$ be  operator-valued orthonormal sequences in  $\mathcal{B}(\mathcal{H},\mathcal{H}_0)$, 	$\{x_n\}_n$, $\{y_n\}_n$ be   sequences in $\mathcal{H}_0$ such that $\{\|x_n\|\|y_n\|\}_n \in \ell^\infty(\mathbb{N})$, $\inf_{n\in \mathbb{N}}\|x_n\|\|y_n\|>0$, and let $\{\lambda_n\}_n$ be  a sequence of scalars. Then the family 
  \begin{align*}
  \{\lambda_n\langle h, B_n^*y_n\rangle A^*_nx_n\}_n
  \end{align*}
  is summable for every $h \in \mathcal{H}$ if and only if $\{\lambda_n\}_n$ is bounded. Whenever $\{\lambda_n\}_n$ is bounded, the map $ \mathcal{H} \ni h \mapsto \sum_{n=1}^{\infty}\lambda_n (A^*_nx_n\otimes \overline{B^*_ny_n})h \in \mathcal{H}$ is a well-defined bounded linear operator with norm at most $\|\{\lambda_n\}_n\|_\infty\sup_{n\in \mathbb{N}}\|x_n\|\|y_n\|.$
 \end{theorem}
 \begin{proof}
 $(\Leftarrow)$ Follows from Theorem \ref{DEFINITIONEXISTENCE} (since an orthonormal operator-valued sequence is an operator-valued Bessel sequence). Note that in this case we can take $a=b=1$.
 
 $(\Rightarrow)$ Let us suppose that $\{\lambda_n\}_n$ is not bounded. Then we can extract a subsequence $\{\lambda_{n_k}\}_{k=1}^\infty$ from $\{\lambda_n\}_n$ such that $|\lambda_{n_k}|\geq k, \forall k \in \mathbb{N}$. For each $m\in \mathbb{N}$, define $ T_m: \mathcal{H} \ni h \mapsto \sum_{k=1}^{m}\lambda_{n_k}\langle h, B^*_{n_k}y_{n_k} \rangle A^*_{n_k}x_{n_k}   \in \mathcal{H}$. Also define $ T: \mathcal{H} \ni h \mapsto \sum_{k=1}^{\infty}\lambda_{n_k}\langle h, B^*_{n_k}y_{n_k} \rangle A^*_{n_k}x_{n_k}   \in \mathcal{H}$. Then for all $r,s \in \mathbb{N}$, $r\leq s$,  since $n_k\geq k , \forall k \in \mathbb{N}$, 
 
 \begin{align*}
 \left\| \sum_{k=r}^s \lambda_{n_k}\langle h, B^*_{n_k}y_{n_k} \rangle A^*_{n_k}x_{n_k}\right\|^2&=\left\langle \sum_{k=r}^s \lambda_{n_k}\langle h, B^*_{n_k}y_{n_k} \rangle A^*_{n_k}x_{n_k}, \sum_{j=r}^s \lambda_{n_j}\langle h, B^*_{n_j}y_{n_j} \rangle A^*_{n_l}x_{n_l}\right \rangle \\
  &=\sum_{k=r}^{s}|\lambda_{n_k}|^2|\langle h, B^*_{n_k}y_{n_k} \rangle|^2 \|x_{n_k}\|^2 \leq  \sum_{k=r}^{s}|\lambda_{k}|^2|\langle h, B^*_{k}y_{k} \rangle|^2 \|x_{k}\|^2\\
  &=\left\|\sum_{k=r}^{s}\lambda_{k}\langle h, B^*_{k}y_{k} \rangle A^*_{k}x_{k} \right\|^2
 \end{align*}
 which is convergent (by assumption). Therefore $T$ is well-defined bounded linear operator. Further, for each fixed $h\in \mathcal{H}$, 
 \begin{align*}
 \|T_mh-Th\|^2&=\left\|\sum_{k=1}^{m}\lambda_{n_k}\langle h, B^*_{n_k}y_{n_k} \rangle A^*_{n_k}x_{n_k}-\sum_{k=1}^{\infty}\lambda_{n_k}\langle h, B^*_{n_k}y_{n_k} \rangle A^*_{n_k}x_{n_k} \right\|^2\\
 &= \left\|\sum_{k=m+1}^{\infty}\lambda_{n_k}\langle h, B^*_{n_k}y_{n_k} \rangle A^*_{n_k}x_{n_k} \right\|^2= \sum_{k=m+1}^{\infty}|\lambda_{n_k}|^2|\langle h, B^*_{n_k}y_{n_k} \rangle|^2 \|x_{n_k}\|^2\\
 &\leq  \sum_{k=m+1}^{\infty}|\lambda_{k}|^2|\langle h, B^*_{k}y_{k} \rangle|^2 \|x_{k}\|^2 =\left\|\sum_{k=m+1}^{\infty}\lambda_{k}\langle h, B^*_{k}y_{k} \rangle A^*_{k}x_{k} \right\|^2\rightarrow 0 \text{ as } m \rightarrow \infty .
 \end{align*} 
 Hence $T_m \rightarrow T$ pointwise which says that $\{T_m\}_{m=1}^\infty$ is bounded pointwise. Now Uniform Boundedness Principle says that there exists $R>0$ such that $\sup_{m \in \mathbb{N}}\|T_m\|\leq R$. Next, for each $m \in \mathbb{N}$, using orthonormality of $B_n$'s,
 
 \begin{align*}
 T_m\left(\frac{B^*_{n_m}y_{n_m}}{\|y_{n_m}\|}\right)&=\sum_{k=1}^{m}\lambda_{n_k}\left\langle \frac{B^*_{n_m}y_{n_m}}{\|y_{n_m}\|}, B^*_{n_k}y_{n_k} \right\rangle A^*_{n_k}x_{n_k} \\
 &=\lambda_{n_m}\|y_{n_m}\|A^*_{n_m}x_{n_m}
 \end{align*} 
 (the condition $\inf_{n\in \mathbb{N}}\|x_n\|\|y_n\|>0$ says that none of the $y_n$'s equals zero). Previous equation along with the observation 
\begin{align*}
\left\|\frac{B^*_{n_m}y_{n_m}}{\|y_{n_m}\|}\right\|\leq \frac{\|B^*_{n_m}\|\|y_{n_m}\|}{\|y_{n_m}\|}=\frac{1.\|y_{n_m}\|}{\|y_{n_m}\|}=1
\end{align*}
gives $R\geq\sup_{m \in \mathbb{N}}\|T_m\|\geq \|T_m\|\geq|\lambda_{n_m}|\|y_{n_m}\|\|A^*_{n_m}x_{n_m}\| =|\lambda_{n_m}|\|y_{n_m}\|\|x_{n_m}\| \geq m\|y_{n_m}\|\|x_{n_m}\|\geq m\inf_{n\in \mathbb{N}}\|y_n\|\|x_n\|$, $\forall m \in \mathbb{N} $ $\Rightarrow $ $m \leq \frac{R}{\inf_{n\in \mathbb{N}}\|y_n\|\|x_n\|}$, $\forall m \in \mathbb{N}$ which is a contradiction.
\end{proof}
\begin{corollary}
Theorem \ref{CHARACTERIZATIONRESULTORTHO}  holds by replacing the condition  $\inf_{n\in \mathbb{N}}\|x_n\|\|y_n\|>0$ with one of the following conditions.
\begin{enumerate}[\upshape(i)]
\item $\inf_{n\in \mathbb{N}}|\langle x_n,y_n\rangle| >0$.
\item $\inf_{n\in \mathbb{N}}\|x_n\|\inf_{n\in \mathbb{N}}\|y_n\|>0.$
\end{enumerate} 
\end{corollary}
 \begin{proof}
 We apply Theorem \ref{CHARACTERIZATIONRESULTORTHO} by noting 
 $$ \inf_{n\in \mathbb{N}}\|x_n\|\|y_n\|\geq\inf_{n\in \mathbb{N}}|\langle x_n,y_n\rangle| >0,$$ 
  $$ \inf_{n\in \mathbb{N}}\|x_n\|\|y_n\|\geq\inf_{n\in \mathbb{N}}\|x_n\|\inf_{n\in \mathbb{N}}\|y_n\|>0.$$
 \end{proof}

 \begin{definition}\label{DEFINITION}
 Let $\{A_n\}_n$ and $\{B_n\}_n$ be  operator-valued Bessel sequences in   $\mathcal{B}(\mathcal{H},\mathcal{H}_0)$ and  let $\{x_n\}_n$, $\{y_n\}_n$ be   sequences in $\mathcal{H}_0$ such that $\{\|x_n\|\|y_n\|\}_n \in \ell^\infty(\mathbb{N})$. For $\{\lambda_n\}_n \in \ell^\infty(\mathbb{N})$, the multiplier for $\{A_n\}_n$ and  $\{B_n\}_n$ is defined as  the operator
 \begin{align}\label{GENERALOPERATOR}
 M_{\lambda,A,B,x,y}\coloneqq \sum_{n=1}^{\infty}\lambda_n (A^*_nx_n\otimes \overline{B^*_ny_n}).
 \end{align}
 \end{definition} 
 \begin{remark}
 Let $\mathcal{H}_0=\mathbb{K}$ and $\{e_n\}_n$, $\{f_n\}_n$ be Bessel (resp. orthonormal) sequences in $\mathcal{H}$. Define $x_n\coloneqq1, y_n\coloneqq1, A_n: \mathcal{H} \ni h \mapsto \langle h, e_n\rangle \in \mathbb{K}, B_n: \mathcal{H} \ni h \mapsto \langle h, f_n\rangle \in \mathbb{K}, \forall n \in \mathbb{N}$. Then  $\sum_{n=1}^{\infty}\lambda_n (A^*_nx_n\otimes \overline{B^*_ny_n})=\sum_{n=1}^{\infty}\lambda_n (e_n\otimes \overline{f_n})$. Thus Definition    \ref{DEFINITION} reduces to the operator of the form $\sum_{n=1}^{\infty}\lambda_n (e_n\otimes \overline{f_n})$, considered by Balazs \cite{BALAZS3} (resp.  Schatten and von Neumann \cite{SCHATTEN}).
 \end{remark}
 Following theorem collects various properties of $M_{\lambda,A,B,x,y}$.
 \begin{theorem}\label{PROPERTIESOFM}
 Let  $\{A_n\}_n$, $\{B_n\}_n$, $\{C_n\}_n$, $\{D_n\}_n$ be  operator-valued Bessel sequences in  $\mathcal{B}(\mathcal{H},\mathcal{H}_0)$ with bounds $a, b,c, d$, respectively, $\{\lambda_n\}_n, \{\mu_n\}_n \in \ell^\infty(\mathbb{N})$, $\alpha \in \mathbb{K}$  and let $\{x_n\}_n$,  $\{y_n\}_n$, $\{z_n\}_n$, $\{v_n\}_n$, $\{z_n\}_n$ be sequences in  $ \mathcal{H}_0$ such that $\{\|x_n\|\|y_n\|\}_n$, $\{\|y_n\|\|z_n\|\}_n$, $\{\|x_n\|\|z_n\|\}_n$,  $\{\|z_n\|\|v_n\|\}_n \in \ell^\infty(\mathbb{N})$. Then
 \begin{enumerate}[\upshape(i)]
 \item   $M_{\lambda,A,B, x,y}^*=M_{\overline{\lambda},B,A, y,x}$, where $\overline{\lambda}\coloneqq\{\overline{\lambda_n}\}_n$. In particular, if $\lambda$ is real valued, then $M_{\lambda,A,A, x,x}$ is self-adjoint.
 \item If  $\{\lambda _n\|y_n\|\}_n \in \ell^\infty(\mathbb{N})$, $\{A_n\}_n$ is orthonormal and  $\{B_n\}_n$ is orthogonal, then  $M_{\lambda,A,B,x,y}M_{\lambda,A,B,x,y}^*=M_{\mu,A,A,x,x} $, where $\mu\coloneqq\{|\lambda _n|^2\|B_n^*y_n\|^2\}_n$. In this case, if $x_n \neq0, \forall n \in \mathbb{N}$, then  $(M_{\lambda,A,B,x,y}M_{\lambda,A,B,x,y}^*)^{1/2}$ $=M_{\sqrt{\mu},A,A,x,x} $, where $\sqrt{\mu}\coloneqq\{|\lambda _n|\frac{\|B_n^*y_n\|}{\|x_n\|}\}_n$.
 \item If  $\{\lambda _n\|x_n\|\}_n \in \ell^\infty(\mathbb{N})$, $\{A_n\}_n$ is orthogonal and $\{B_n\}_n$ is orthonormal, then  $M_{\lambda,A,B,x,y}^*M_{\lambda,A,B,x,y}=M_{\gamma,B,B,y,y} $, where $\gamma\coloneqq\{|\lambda _n|^2\|A_n^*x_n\|^2\}_n$. In this case, if $y_n \neq0, \forall n \in \mathbb{N}$, then  $(M_{\lambda,A,B,x,y}^*M_{\lambda,A,B,x,y})^{1/2}$ $=M_{\sqrt{\gamma},B,B,y,y} $, where $\sqrt{\gamma}\coloneqq\{|\lambda _n|\frac{\|A_n^*x_n\|}{\|y_n\|}\}_n$.
 \item If $\langle A^*_kx_k, B_n^*y_n\rangle =0 , \forall k, n \in \mathbb{N}$ with $ k\neq n$, then for all $k \in \mathbb{N}$,
 \begin{align*}
 M_{\lambda,A,B,x,y}^k=\sum_{n=1}^{\infty}\lambda_n^k\langle A_n^*x_n, B_n^*y_n\rangle ^{k-1} (A^*_nx_n\otimes \overline{B^*_ny_n}).
 \end{align*}
 In particular, if $\{A_n\}_n$ is orthogonal, then  $M_{\lambda,A,A,x,y}^k=\sum_{n=1}^{\infty}\lambda_n^k\langle A_n^*x_n, A_n^*y_n\rangle ^{k-1} (A^*_nx_n\otimes \overline{A^*_ny_n}), \forall k \in \mathbb{N}.$
 \item $ M_{\alpha\lambda,A,B,x,y}= M_{\lambda,A,B,\alpha x,y}= M_{\lambda,A,\alpha B,x,y}=\alpha M_{\lambda,A,B,x,y}$,  $M_{\lambda,\alpha A,B,x,y}=M_{\lambda,A,B,x,\alpha y}=\overline{\alpha}M_{\lambda,A,B,x,y} $. 
 \item $ M_{\lambda+\mu,A,B,x,y}= M_{\lambda,A,B,x,y}+ M_{\mu,A,B,x,y}$. 
 \item $ M_{\lambda,A+C, B,x,y}= M_{\lambda,A,B,x,y}+ M_{\lambda,C,B,x,y}$.
 \item $ M_{\lambda,A,B+C,x,y}= M_{\lambda,A,B,x,y}+ M_{\lambda,A,C,x,y}$.
  \item $ M_{\lambda,A,B,x+y,z}= M_{\lambda,A,B,x,z}+ M_{\mu,A,B,y,z}$.
  \item $ M_{\lambda,A,B,x,y+z}= M_{\lambda,A,B,x,y}+ M_{\mu,A,B,x,z}$.
  \item If $\{A_n\}_n$ is orthogonal, then $\|M_{\lambda\mu,A,B,x,y}\|\leq \min\{\sup_{n\in \mathbb{N}}|\lambda_n|\|M_{\mu,A,B,x,y}\|,\sup_{n\in \mathbb{N}}|\mu_n|\|M_{\lambda,A,B,x,y}\| \}$, where $\lambda \mu\coloneqq\{\lambda_n\mu_n\}_n$.
 \item (Symbolic calculus) If $\langle A^*_kx_k, B_n^*y_n\rangle =0 , \forall k, n \in \mathbb{N}$ with $ k\neq n$, then $M_{\lambda,A,B,x,y}M_{\mu,A,B,x,y}= M_{\nu,A,B,x,y}$, where $\nu \coloneqq \{\lambda_n\mu_n\langle A_n^*x_n, B_n^*y_n\rangle \}_n$. Moreover, if  $A_n^*x_n=x,B_n^*y_n=y,\forall n \in \mathbb{N}$, then $M_{\lambda,A,B,x,y}M_{\mu,A,B,x,y}=\langle x, y\rangle M_{\lambda\mu,A,B,x,y}.$ In particular, if $\langle x, y\rangle =1$, then $ M_{\lambda,A,B,x,y}M_{\mu,A,B,x,y}=M_{\lambda\mu,A,B,x,y}.$
 \item If $\{A_n\}_n$ is orthogonal, then $M_{\lambda,A,A, x,x}$ is normal.
\item If  $\{T_n\}_n$ in $\mathcal{B}(\mathcal{H}_0)$ is such that $\sup_{n\in \mathbb{N}}\|T_n\|< \infty$, then $ M_{\lambda,A,TB,x,y}= M_{\lambda,A,B,x,T^*y}$, where $T^*y\coloneqq\{T_n^*y_n\}_n$. 
  \item If $S \in \mathcal{B}(\mathcal{H})$, then $M_{\lambda,A,BS,x,y}=M_{\lambda,A,B,x,y}S.$
  \item If $\{T_n\}_n$ in $\mathcal{B}(\mathcal{H}_0)$ is such that $\sup_{n\in \mathbb{N}}\|T_n\|< \infty$, then $ M_{\lambda,TA,B,x,y}= M_{\lambda,A,B,T^*x,y}$, where $T^*y\coloneqq\{T_n^*x_n\}_n$. 
 \item If $S \in \mathcal{B}(\mathcal{H})$, then $M_{\lambda,AS,B,x,y}=S^*M_{\lambda,A,B,x,y}$.
\item If   $\{T_n\}_n$ in $\mathcal{B}(\mathcal{H}_0)$ is such that $\sup_{n\in \mathbb{N}}\|T_n\|< \infty$, then $ M_{\lambda,A,B,x,Ty}= M_{\lambda,A,T^*B,x,y}$, where $T^*B\coloneqq\{T_n^*B\}_n$.
 \item If  $\{T_n\}_n$ in $\mathcal{B}(\mathcal{H}_0)$ is such that $\sup_{n\in \mathbb{N}}\|T_n\|< \infty$, then $ M_{\lambda,A,B,Tx,y}= M_{\lambda,T^*A,B,x,y}$.
 \item If $\langle C_k^*z_k,B_n^*y_n \rangle=0,\forall k, n \in \mathbb{N}$ with $ k\neq n$, then $M_{\lambda,A,B,x,y}M_{\mu,C,D,z,v}=M_{\lambda\mu\langle Cz, Dy\rangle,A,D,x,v}$, where $\lambda\mu\langle Cz, Dy\rangle\coloneqq\{\lambda_n\mu_n\langle C_n^*z_n, D_n^*y_n\rangle\}$. In particular, if $\{B_n\}_n$ is orthogonal, then $M_{\lambda,A,B,x,y}M_{\mu,B,D,z,v}$ $=M_{\lambda\mu\langle Bz, Dy\rangle,A,D,x,v}$.
 \end{enumerate}
 \end{theorem} 
 \begin{proof}
 \begin{enumerate}[\upshape(i)]
 \item $M_{\lambda,A,B,x,y}^*= \sum_{n=1}^{\infty}(\lambda_n (A^*_nx_n\otimes \overline{B^*_ny_n}))^*=\sum_{n=1}^{\infty}\overline{\lambda_n}(B^*_ny_n\otimes \overline{A^*_nx_n})$.
 \item 
 \begin{align*}
 M_{\lambda,A,B,x,y}M_{\lambda,A,B,x,y}^*&=\sum_{n=1}^{\infty}\lambda_n(A_n^*x_n\otimes \overline{B_n^*y_n})\left(\sum_{k=1}^{\infty}\overline{\lambda_k} (B_k^*y_k\otimes \overline{A_k^*x_k})\right)\\
 &=\sum_{n=1}^{\infty}\lambda_n\sum_{k=1}^{\infty}\overline{\lambda_k} (A_n^*x_n\otimes \overline{B_n^*y_n})(B_k^*y_k\otimes \overline{A_k^*x_k})\\
 &=\sum_{n=1}^{\infty}\lambda_n\sum_{k=1}^{\infty}\overline{\lambda_k}\langle B_k^*y_k, B_n^*y_n\rangle (A_n^*x_n \otimes \overline{A_k^*x_k})\\
 &=\sum_{n=1}^{\infty}\lambda_n\overline{\lambda_n}\langle B_n^*y_n, B_n^*y_n\rangle (A_n^*x_n \otimes \overline{A_n^*x_n})\\
 &=\sum_{n=1}^{\infty}|\lambda_n|^2 \|B_n^*y_n\|^2 (A_n^*x_n \otimes \overline{A_n^*x_n}).
 \end{align*}
 Define $T \coloneqq M_{\sqrt{\mu},A,A,x,x}=\sum_{n=1}^{\infty}|\lambda_n| \frac{\|B_n^*y_n\|}{\|x_n\|} (A_n^*x_n \otimes \overline{A_n^*x_n})$. Then 
 
 \begin{align*}
 T^2&=\sum_{n=1}^{\infty}|\lambda_n| \frac{\|B_n^*y_n\|}{\|x_n\|}(A_n^*x_n \otimes \overline{A_n^*x_n})\left(\sum_{k=1}^{\infty}|\lambda_k| \frac{\|B_k^*y_k\|}{\|x_k\|} (A_k^*x_k \otimes \overline{A_k^*x_k})\right)\\
 &=\sum_{n=1}^{\infty}|\lambda_n|\frac{\|B_n^*y_n\|}{\|x_n\|}\sum_{k=1}^{\infty}|\lambda_k|\frac{\|B_k^*y_k\|}{\|x_k\|}\langle A_k^*x_k, A_n^*x_n\rangle (A_n^*x_n \otimes \overline{A_k^*x_k})\\
 &=\sum_{n=1}^{\infty}|\lambda_n|^2\frac{\|B_n^*y_n\|^2}{\|x_n\|^2}\|x_n\|^2(A_n^*x_n \otimes \overline{A_n^*x_n})=M_{\lambda,A,B,x,y}M_{\lambda,A,B,x,y}^*.
 \end{align*}
 Therefore $T=(M_{\lambda,A,B,x,y}M_{\lambda,A,B,x,y}^*)^{1/2} $.
 \item Similar to the proof of (ii).
 \item The proof is by induction. When $k=2,$
 \begin{align*}
 M_{\lambda,A,B,x,y}^2&=\sum_{n=1}^{\infty}\lambda_n(A_n^*x_n\otimes \overline{B_n^*y_n})\left(\sum_{k=1}^{\infty}\lambda_k (A_k^*x_k\otimes \overline{B_k^*y_k})\right)\\
 &=\sum_{n=1}^{\infty}\lambda_n^2\langle A_n^*x_n, B_n^*y_n\rangle ^{1} (A^*_nx_n\otimes \overline{B^*_ny_n}).
 \end{align*}
 Assume the result is true for $m$. Then 
 
 \begin{align*}
 M_{\lambda,A,B,x,y}^{m+1}&=M_{\lambda,A,B,x,y}^{1}M_{\lambda,A,B,x,y}^{m}\\
 &=\sum_{n=1}^{\infty}\lambda_n(A_n^*x_n\otimes \overline{B_n^*y_n})\left(\sum_{k=1}^{\infty}\lambda_k^m\langle A_k^*x_k, B_k^*y_k\rangle^{m-1} (A_k^*x_k\otimes \overline{B_k^*y_k})\right)\\
 &=\sum_{n=1}^{\infty}\lambda_n\sum_{k=1}^{\infty}\lambda_k^m\langle A_k^*x_k, B_k^*y_k\rangle^{m-1}\langle A_k^*x_k, B_n^*y_n\rangle(A_n^*x_n\otimes \overline{B_k^*y_k})\\
 &=\sum_{n=1}^{\infty}\lambda_n^{m+1}\langle A_n^*x_n, B_n^*y_n\rangle ^{m} (A^*_nx_n\otimes \overline{B^*_ny_n}).
 \end{align*}
 Hence the conclusion.
 \item This is clear. 
 \item $ M_{\lambda+\mu,A,B,x,y}=\sum_{n=1}^{\infty}(\lambda_n+\mu_n) (A^*_nx_n\otimes \overline{B^*_ny_n})=\sum_{n=1}^{\infty}\lambda_n (A^*_nx_n\otimes \overline{B^*_ny_n})+\sum_{n=1}^{\infty}\mu_n (A^*_nx_n\otimes \overline{B^*_ny_n})= M_{\lambda,A,B,x,y}+ M_{\mu,A,B,x,y}$.
 \item We first note that the sum of two operator-valued Bessel sequences is again an   operator-valued Bessel sequence. In fact,
 
 \begin{align*}
 \left(\sum_{n=1}^\infty\|(A_n+C_n)h\|^2 \right)^\frac{1}{2}&=\left(\sum_{n=1}^\infty\|A_nh+C_nh\|^2 \right)^\frac{1}{2}\leq \left(\sum_{n=1}^\infty(\|A_nh\|+\|C_nh\|)^2 \right)^\frac{1}{2}\\
 &\leq \left(\sum_{n=1}^\infty\|A_nh\|^2 \right)^\frac{1}{2}+\left(\sum_{n=1}^\infty\|C_nh\|^2 \right)^\frac{1}{2}\\
 &\leq a\|h\|+c\|h\|=(a+c)\|h\|,~ \forall h \in \mathcal{H}.
 \end{align*}
 Now $ M_{\lambda,A+C,B,x,y}h=\sum_{n=1}^{\infty}\lambda_n\langle h,B^*_ny_n \rangle (A_n^*+C_n^*)x_n= M_{\lambda,A,B,x,y}h+ M_{\lambda,C,B,x,y}h, \forall h \in \mathcal{H}$.
 \item Similar to (vii).
 \item This follows from the linearity of $A_n^*$ for all $n \in \mathbb{N}$.
 \item This follows from the linearity of $B_n^*$ for all $n \in \mathbb{N}$ and linearity of inner product.
 \item For all $h \in \mathcal{H}$, $\|M_{\lambda\mu,A,B,x,y}h\|^2=\|\sum_{n=1}^{\infty}\lambda_n\mu_n(A_n^*x_n\otimes \overline{B_n^*y_n})h\|^2=\langle \sum_{n=1}^{\infty}\lambda_n\mu_n\langle h, B_n^*y_n\rangle A_n^*x_n, $ $\sum_{k=1}^{\infty}\lambda_k\mu_k\langle h,  B_k^*y_k\rangle A_k^*x_k\rangle=\sum_{n=1}^{\infty}|\lambda_n\mu_n|^2|\langle h, B_n^*y_n\rangle|^2\|x_n\|^2\leq \sup_{n\in \mathbb{N}}|\lambda_n| \sum_{n=1}^{\infty}|\mu_n|^2|\langle h, B_n^*y_n\rangle|^2\|x_n\|^2= \sup_{n\in \mathbb{N}}|\lambda_n|^2\|M_{\mu,A,B,x,y}h\|^2\leq \sup_{n\in \mathbb{N}}|\lambda_n|^2\|M_{\mu,A,B,x,y}\|^2\|h\|^2$. Similarly $\|M_{\lambda\mu,A,B,x,y}h\|^2\leq$ \\
$\sup_{n\in \mathbb{N}}|\mu_n|^2\|M_{\lambda,A,B,x,y}\|^2\|h\|^2$.
 \item Note that $\{\lambda_n\mu_n\langle A_n^*x_n, B_n^*y_n \rangle \}_n \in \ell^\infty(\mathbb{N})$. In fact, 
 
 $ \sup_{n\in \mathbb{N}}|\lambda_n\mu_n\langle A_n^*x_n, B_n^*y_n \rangle|\leq \sup_{n\in \mathbb{N}}|\lambda_n|\sup_{n\in \mathbb{N}}|\mu_n|\sup_{n\in \mathbb{N}}\|A_n\|\sup_{n\in \mathbb{N}}\|B_n\|\sup_{n\in \mathbb{N}}\|x_n\|\|y_n\|<\infty.$
 Then 
 \begin{align*}
 M_{\lambda,A,B,x,y}M_{\mu,A,B,x,y}h&=M_{\lambda,A,B,x,y}\left(\sum_{n=1}^{\infty}\mu_n\langle h,B^*_ny_n \rangle A^*_nx_n\right)\\
 &=\sum_{n=1}^{\infty}\mu_n\langle h,B^*_ny_n \rangle M_{\lambda,A,B,x,y}(A_n^*x_n)\\
 &=\sum_{n=1}^{\infty}\mu_n\langle h,B^*_ny_n \rangle\sum_{k=1}^{\infty}\lambda_k\langle A_n^*x_n, B^*_ky_k\rangle A^*_kx_k\\
 &=\sum_{n=1}^{\infty}\mu_n\lambda_n\langle h,B^*_ny_n \rangle \langle A^*_nx_n, B^*_ny_n\rangle A^*_nx_n= M_{\nu,A,B,x,y}h.
 \end{align*}
 \item Comes from (i) and symbolic calculus.
\item First we verify that $\{T_nB_n\}_n$ is an   operator-valued Bessel sequence in  $\mathcal{B}(\mathcal{H},\mathcal{H}_0)$: $ \left(\sum_{n=1}^\infty\|T_nB_nh\|^2 \right)^\frac{1}{2}\leq \left(\sum_{n=1}^\infty\|T_n\|^2\|B_nh\|^2 \right)^\frac{1}{2}
 \leq \sup_{n \in \mathbb{N}}\|T_n\|\left(\sum_{n=1}^\infty\|B_nh\|^2 \right)^\frac{1}{2}\leq \sup_{n \in \mathbb{N}}\|T_n\|b\|h\|, \forall h \in \mathcal{H}.$ We next see $\sup_{n\in\mathbb{N}}\|x_n\|\|T_n^*y_n\|\leq \sup_{n\in\mathbb{N}}\|T_n\|\sup_{n\in\mathbb{N}}\|x_n\|\|y_n\|<\infty$. Now $ M_{\lambda,A,TB,x,y}h=\sum_{n=1}^{\infty}\lambda_n\langle h,(T_nB_n)^*y_n \rangle A_n^*x_n$ $=\sum_{n=1}^{\infty}\lambda_n\langle h,B_n^*(T_n^*y_n) \rangle A_n^*x_n= M_{\lambda,A,B,x,T^*y}h, \forall h \in \mathcal{H}$.
 \item Note that $\{B_nS\}_n$ is   an operator-valued Bessel sequence. In fact, $\sum_{n=1}^{\infty}\|B_nSh\|^2\leq b \|Sh\|^2 \leq b\|S\|^2\|h\|^2, \forall h \in \mathcal{H}$. Next, $M_{\lambda,A,BS,x,y}h=\sum_{n=1}^{\infty}\lambda_n\langle h, (B_nS)^*y_n\rangle A_n^*x_n=\sum_{n=1}^{\infty}\lambda_n\langle Sh, B_n^*y_n\rangle A_n^*x_n$ $=M_{\lambda,A,B,x,y}Sh, \forall h \in \mathcal{H}.$
 \item $ M_{\lambda,TA,B,x,y}h=\sum_{n=1}^{\infty}\lambda_n\langle h,B_n^*y_n \rangle (T_nA_n)^*x_n=\sum_{n=1}^{\infty}\lambda_n\langle h,B_n^*y_n \rangle A_n^*(T_n^*x_n)= M_{\lambda,A,B,T^*x,y}h, \forall h \in \mathcal{H}$.
 \item  $M_{\lambda,AS,B,x,y}h=\sum_{n=1}^{\infty}\lambda_n\langle h,B_n^*y_n \rangle (A_nS)^*x_n=S^*(\sum_{n=1}^{\infty}\lambda_n\langle h,B_n^*y_n \rangle A_n^*x_n)=S^*M_{\lambda,A,B,x,y}h,\forall h \in \mathcal{H}$.
 \item $ M_{\lambda,A,B,x,Ty}=\sum_{n=1}^{\infty}\lambda_n\langle h,B_n^*T_ny_n \rangle A_n^*x_n=\sum_{n=1}^{\infty}\lambda_n\langle h,(T_n^*B_n)^*y_n \rangle A_n^*x_n= M_{\lambda,A,T^*B,x,y}$.
 \item $M_{\lambda,A,B,Tx,y}=\sum_{n=1}^{\infty}\lambda_n\langle h,B_n^*y_n \rangle A_n^*Tx_n=\sum_{n=1}^{\infty}\lambda_n\langle h,B_n^*y_n \rangle (T^*A_n)^*x_n=M_{\lambda,T^*A,B,x,y}$.
 \item $M_{\lambda,A,B,x,y}M_{\mu,C,D,z,v}h=\sum_{n=1}^{\infty}\lambda_n\langle M_{\mu,C,D,z,v}h,B_n^*y_n \rangle A_n^*x_n=\sum_{n=1}^{\infty}\sum_{k=1}^{\infty}\lambda_n\mu_k \langle h, D_k^*v_k\rangle \langle C_k^*z_k,B_n^*y_n \rangle A_n^*x_n$ $=\sum_{n=1}^{\infty}\lambda_n\mu_n\langle C_n^*z_n, D_n^*y_n\rangle\langle h, D_n^*v_n\rangle A_n^*x_n$.
 \end{enumerate}	
 \end{proof} 

 \begin{theorem}
 Let $a $ (resp. $b $) be a Bessel bound for $\{A_n\}_n$ (resp.  $\{B_n\}_n$).
 \begin{enumerate}[\upshape(i)]
 \item If $\{\lambda_n\}_n \in c_0(\mathbb{N})$, then $M_{\lambda,A,B,x,y}$ is a compact operator.
 \item If $\{\lambda_n\}_n \in \ell^1(\mathbb{N})$, then  $\|M_{\lambda,A,B,x,y}\|_{\operatorname{Nuc}}\leq\sqrt{ab}\sup_{n\in \mathbb{N}}\|x_n\|\|y_n\|\|\{\lambda_n\}_n\|_1$.
 \item If $\{\lambda_n\}_n \in \ell^2(\mathbb{N})$ and $\{A_n\}_n$ is orthogonal, then  $M_{\lambda,A,B,x,y}$ is a Hilbert-Schmidt operator with
 
  $\sigma(M_{\lambda,A,B,x,y})\leq \sqrt{ab}\sup_{n\in \mathbb{N}}\|x_n\|\|y_n\|\|\{\lambda_n\}_n\|_2$.
 \end{enumerate}
 \end{theorem}
 \begin{proof}
  \begin{enumerate}[\upshape(i)]
  \item  For $m \in \mathbb{N}$, define $M_{\lambda_m,A,B,x,y}\coloneqq\sum_{n=1}^{m}\lambda_n(A_n^*x_n\otimes B_n^*y_n)$. Then $M_{\lambda_m,A,B,x,y}(\mathcal{H})\subseteq\operatorname{span}\{A_n^*x_n\}_{n=1}^m$ and $\|M_{\lambda_m,A,B,x,y}h-M_{\lambda,A,B,x,y}h\|=\|\sum_{n=m+1}^{\infty}\lambda_n(A_n^*x_n\otimes B_n^*y_n)h\|\leq$ 
  
  $ \sqrt{ab} \sup_{n\in \mathbb{N}}\|x_n\|\|y_n\|\|h\|\sup_{m+1\leq n<\infty }|\lambda_n|$, $ \forall h \in \mathcal{H}$. Hence $\{M_{\lambda_m,A,B,x,y}\}_{m=1}^\infty$ is a sequence of finite rank operators and  $\|M_{\lambda_m,A,B,x,y}-M_{\lambda,A,B,x,y}\|\leq\sqrt{ab} \sup_{n\in \mathbb{N}}\|x_n\|\|y_n\|\sup_{m+1\leq n<\infty }|\lambda_n|$ $\rightarrow$ $0$ as $m \rightarrow \infty$. Thus $M_{\lambda_m,A,B,x,y}$ is compact.
  \item Define $f_n: \mathcal{H}\ni h \mapsto \langle h,B_n^*y_n \rangle \in \mathbb{K}$ for each $n \in \mathbb{N}$. Then $f_n\in  \mathcal{H}^*$ and $ \|f_n\|=\|B_n^*y_n\|$ for all $n \in \mathbb{N}$. This yields  $\|M_{\lambda,A,B,x,y}\|_{\operatorname{Nuc}}\leq\sum_{n=1}^{\infty}\|f_n\|\|\lambda_nA_n^*x_n\|=\sum_{n=1}^{\infty}|\lambda_n|\|A_n^*x_n\|\|B_n^*y_n\|\leq\sum_{n=1}^{\infty}|\lambda_n|\|A_n^*\|\|x_n\|\|B_n^*\|\|y_n\|\leq \sum_{n=1}^{\infty}|\lambda_n|\|x_n\|\sqrt{ab}\|y_n\|\leq\sqrt{ab}\sum_{n=1}^{\infty}|\lambda_n|\sup_{n\in \mathbb{N}}\|x_n\|\|y_n\| =$ 
  
  $\sqrt{ab}\sup_{n\in \mathbb{N}}\|x_n\|\|y_n\|\|\{\lambda_n\}_n\|_1$.
  \item Let $\{e_n\}_n$ be an orthonormal basis for $\mathcal{H}$. Then $\sigma(M_{\lambda,A,B,x,y})^2= \sum_{n=1}^{\infty}\|M_{\lambda,A,B,x,y}e_n\|^2=\sum_{n=1}^{\infty}\|\sum_{k=1}^{\infty}\lambda_k(A_k^*x_k\otimes \overline{B_k^*y_k})\|^2=\sum_{n=1}^{\infty}\langle \sum_{k=1}^{\infty}\lambda_k(A_k^*x_k\otimes \overline{B_k^*y_k})e_n,\sum_{r=1}^{\infty}\lambda_r(A_r^*x_r\otimes \overline{B_r^*y_r})e_n \rangle=\sum_{n=1}^{\infty}\sum_{k=1}^{\infty}|\lambda_k|^2\|A_k^*x_k\|^2|\langle e_n, B_k^*y_k\rangle|^2 =\sum_{k=1}^{\infty}\sum_{n=1}^{\infty}|\lambda_k|^2\|A_k^*x_k\|^2|\langle e_n, B_k^*y_k\rangle|^2=$
  $\sum_{k=1}^{\infty}|\lambda_k|^2\|A_k^*x_k\|^2$ $\sum_{n=1}^{\infty}|\langle e_n, B_k^*y_k\rangle|^2=\sum_{k=1}^{\infty}|\lambda_k|^2\|A_k^*x_k\|^2\|B_k^*y_k\|^2 \leq \sum_{k=1}^{\infty}|\lambda_k|^2\|A_k^*\|^2\|x_k\|^2\|B_k^*\|^2\|y_k\|^2\leq ab \sup_{n\in \mathbb{N}}\|x_n\|^2\|y_n\|^2$ $\|\{\lambda_n\}_n\|^2_2$.
  \end{enumerate}	
 \end{proof}
 We derive  continuity of multiplier in the next proposition.
  \begin{proposition}
  Let $\lambda_n^{(k)}=\{\lambda_n^{(k)}\}_n$. If $\{\lambda_n^{(k)}\}_{k=1}^\infty $ converges to $\{\lambda_n\}_n $ as $k \rightarrow \infty $ in	
  \begin{enumerate}[\upshape(i)]
 \item  $\ell^\infty(\mathbb{N})$,  then $\{M_{\lambda^{(k)},A,B,x,y}\}_{k=1}^\infty$ converges to  $M_{\lambda,A,B,x,y}$ as $k \rightarrow \infty $ in the operator-norm. 
 \item  $\ell^1(\mathbb{N})$,   then $\{M_{\lambda^{(k)},A,B,x,y}\}_{k=1}^\infty$ converges to  $M_{\lambda,A,B,x,y}$ as $k \rightarrow \infty $ in the nuclear-norm.
 \item $\ell^2(\mathbb{N})$ and $\{A_n\}_n$ is orthogonal,   then $\{M_{\lambda^{(k)},A,B,x,y}\}_{k=1}^\infty$ converges to  $M_{\lambda,A,B,x,y}$ as $k \rightarrow \infty $ in the Hilbert-Schmidt norm.
 \end{enumerate}
 \end{proposition}
 \begin{proof}
 Let $a$ (resp. $b$) be a Bessel bound for $\{A_n\}_n $ (resp. $\{B_n\}_n $).
 \begin{enumerate}[\upshape(i)]
 \item  For each $h \in \mathcal{H}$, $\|M_{\lambda^{(k)},A,B,x,y}h- M_{\lambda,A,B,x,y}h\|=\|\sum_{n=1}^{\infty}(\lambda_n^{(k)}-\lambda_n)(A_n^*x_n\otimes \overline{B_n^*y_n})h\|\leq \sqrt{ab} \sup_{n\in \mathbb{N}}|\lambda_n^{(k)}-\lambda_n|\|h\|\sup_{n\in \mathbb{N}}\|x_n\|\|y_n\|$. Therefore $ \|M_{\lambda^{(k)},A,B,x,y}-M_{\lambda,A,B,x,y}\|\leq\sqrt{ab}\sup_{n\in \mathbb{N}}\|x_n\|\|y_n\|\sup_{n\in \mathbb{N}}|\lambda_n^{(k)}-\lambda_n|\rightarrow 0 $ as $k \rightarrow \infty $.
 \item $\|M_{\lambda^{(k)},A,B,x,y}- M_{\lambda,A,B,x,y}\|_{\operatorname{Nuc}}=\|\sum_{n=1}^{\infty}(\lambda^{(k)}_n-\lambda_n)(A_n^*x_n\otimes \overline{B_n^*y_n})\|_{\operatorname{Nuc}}\leq\sum_{n=1}^{\infty}|\lambda^{(k)}_n-\lambda_n|\|A_n^*x_n\|\|B_n^*y_n\| $ $\leq \sqrt{ab}\sup_{n\in \mathbb{N}}\|x_n\|\|y_n\|\sum_{n=1}^{\infty}|\lambda^{(k)}_n-\lambda_n|\rightarrow 0$ as $k \rightarrow \infty $.
 \item Starting with an orthonormal basis $\{e_n\}_n$ for $\mathcal{H}$, we see that $ \sigma(M_{\lambda^{(k)},A,B,x,y}- M_{\lambda,A,B,x,y})^2=\sigma(\sum_{n=1}^{\infty}(\lambda^{(k)}_n-\lambda_n)(A_n^*x_n\otimes \overline{B_n^*y_n}))^2=\sum_{r=1}^{\infty}\|\sum_{n=1}^{\infty}(\lambda^{(k)}_n-\lambda_n)(A_n^*x_n\otimes \overline{B_n^*y_n})e_r\|^2=\sum_{r=1}^{\infty}\sum_{n=1}^{\infty}|\lambda^{(k)}_n-\lambda_n|^2\|A_n^*x_n\|^2|\langle e_r,B_n^*y_n \rangle |^2=\sum_{n=1}^{\infty}|\lambda^{(k)}_n-\lambda_n|^2\|A_n^*x_n\|^2\|B_n^*y_n\|^2\leq ab \sup_{n\in \mathbb{N}}\|x_n\|^2\|y_n\|^2\sum_{n=1}^{\infty}|\lambda^{(k)}_n-\lambda_n|^2 \rightarrow 0$ as $k \rightarrow \infty $.
 \end{enumerate}	
 \end{proof}
 Following proposition shows that in a special case, the multiplier is bounded below.
 \begin{proposition}
Let $\{A_n\}_n$ be  an operator-valued orthonormal basis in  $\mathcal{B}(\mathcal{H},\mathcal{H}_0)$ and   $\{B_n\}_n$ be  an operator-valued Riesz basis in  $\mathcal{B}(\mathcal{H},\mathcal{H}_0)$. 
 \begin{enumerate}[\upshape(i)]
 \item If $x_n\neq0 \neq y_n, \forall n \in \mathbb{N}$, then the map $ S: \ell^\infty(\mathbb{N})\ni \{\lambda_n\}_n \mapsto M_{\lambda,A,B,x,y} \in \mathcal{B}(\mathcal{H})$ is a well-defined injective bounded linear operator.
 \item There exists a unique invertible $T \in \mathcal{B}(\mathcal{H})$ such that 
 \begin{align*}
  \sup_{g \in \mathcal{H}_0, g \neq 0}\sup_{n\in \mathbb{N}}\frac{|\lambda_n\langle g, y_n\rangle |\|x_n\|}{\|T^{-1}A_n^*g\|}\leq\|M_{\lambda,A,B,x,y}\|\leq \|T\|\sup_{n\in \mathbb{N}} |\lambda_n|\sup_{n\in \mathbb{N}} \|x_n\| \|y_n\|.
 \end{align*}
 In particular, if $x_n\neq 0, \forall n \in \mathbb{N}$ (resp. $y_n\neq 0, \forall n \in \mathbb{N}$), then 
 
 \begin{align*}
 \frac{\sup_{n\in \mathbb{N}}|\lambda_n\langle x_n, y_n\rangle |}{\|T^{-1}\|}\leq\|M_{\lambda,A,B,x,y}\|\quad \left(\text{resp. } \sup_{n\in \mathbb{N}}\frac{|\lambda_n \|x_n\| \|y_n\|}{\|T^{-1}\|}\leq\|M_{\lambda,A,B,x,y}\|\right).
 \end{align*}
 \end{enumerate}
 \end{proposition} 
 \begin{proof}
 Let  $T \in \mathcal{B}(\mathcal{H})$ be the unique invertible operator  such that $B_n=A_nT, \forall n \in \mathbb{N}$, given by Theorem \ref{RIESZBASISCRITERION}.
\begin{enumerate}[\upshape(i)]
\item By using Theorem \ref{DEFINITIONEXISTENCE}, $\|S\{\lambda_n\}_n\|=\| M_{\lambda,A,B,x,y}\|\leq\|T\|\sup_{n\in \mathbb{N}} \|x_n\| \|y_n\|\|\{\lambda_n\}_n\|_\infty$, $\forall \{\lambda_n\}_n \in \ell^\infty(\mathbb{N})$ which implies  $S$ is bounded. From (v) and (vi) in Theorem \ref{PROPERTIESOFM} we see that $S$ is linear. Now suppose $S\{\lambda_n\}_n=0$ for some $\{\lambda_n\}_n \in\ell^\infty(\mathbb{N})$. Then $\sum_{n=1}^{\infty}\lambda_n\langle Th,A^*_ny_n\rangle A_n^*x_n=\sum_{n=1}^{\infty}\lambda_n\langle h,B^*_ny_n\rangle A_n^*x_n =\sum_{n=1}^{\infty}\lambda_n (A^*_nx_n\otimes \overline{B^*_ny_n})h=M_{\lambda,A,B,x,y}h=(S\{\lambda_n\}_n)h=0,\forall h \in \mathcal{H}$. By taking $h=T^{-1}A_k^*y_k, k \in \mathbb{N}$ we get $\lambda_k\|y_k\|^2x_k=A_k(\lambda_k\|y_k\|^2A_k^*x_k)=A_k0=0,\forall k \in \mathbb{N}$. Hence $S$ is injective.
\item Let $n \in \mathbb{N}$ be fixed  and $g \in \mathcal{H}_0$ be nonzero.  We note that $A_n^*g\neq0$. Else $0=A_nA_n^*g=g$, which is forbidden. Upper bound for the operator norm is clear and for lower, 
\begin{align*}
\|M_{\lambda,A,B,x,y}\|&=\sup _{h\in \mathcal{H}, \|h\|\leq 1}\|M_{\lambda,A,B,x,y}h\|=\sup _{h\in \mathcal{H}, \|h\|\leq 1}\left\|\sum_{k=1}^{\infty}\lambda_k\langle Th,A^*_ky_k \rangle A^*_kx_k\right\|\\
&\geq \left\|\sum_{k=1}^{\infty}\lambda_k\langle T\frac{T^{-1}A^*_ng}{\|T^{-1}A^*_ng\|},A^*_ky_k \rangle A^*_kx_k\right\|=\frac{|\lambda_n\langle g, y_n\rangle |\|x_n\|}{\|T^{-1}A_n^*g\|}.
\end{align*}
If $x_n\neq 0, \forall n \in \mathbb{N}$ (resp. $y_n\neq 0, \forall n \in \mathbb{N}$), then we take $g=x_n$ (resp. $g=y_n$) to get 

\begin{align*}
&\frac{|\lambda_n\langle x_n, y_n\rangle |\|x_n\|}{\|T^{-1}A_n^*x_n\|}\geq \frac{|\lambda_n\langle x_n, y_n\rangle |\|x_n\|}{\|T^{-1}\|\|A_n^*x_n\|}=\frac{|\lambda_n\langle x_n, y_n\rangle |\|x_n\|}{\|T^{-1}\|\|x_n\|}=\frac{|\lambda_n\langle x_n, y_n\rangle |}{\|T^{-1}\|}\\
\bigg(\text{resp. } \quad &\frac{|\lambda_n\langle y_n, y_n\rangle |\|x_n\|}{\|T^{-1}A_n^*y_n\|}\geq\frac{|\lambda_n| \|y_n\|^2 \|x_n\|}{\|T^{-1}\|\|A_n^*y_n\|}=\frac{|\lambda_n| \|y_n\| \|x_n\|}{\|T^{-1}\|}\bigg).
\end{align*}
\end{enumerate}	
 \end{proof}

 \section{Generalized Hilbert-Schmidt class}\label{HSCLASSSECTION}
 Let $\{x_n\}_n$ be a  frame for $\mathcal{H}$. In \cite{BALAZS4, FRANKPAULSENTIBALLI}, operators $A\in \mathcal{B}(\mathcal{H})$ satisfying $\sum_{n=1}^{\infty}\|Ax_n\|^2<\infty$ are studied. In Lemma 1.2 of \cite{FRANKPAULSENTIBALLI} it was showed that if $\{x_n\}_n$, $\{y_n\}_n$ are Parseval frames,  then $\sum_{n=1}^{\infty}\|Ax_n\|^2<\infty$ if and only if $\sum_{n=1}^{\infty}\|Ay_n\|^2<\infty$ and in this case, $\sigma(A)^2=\sum_{n=1}^{\infty}\|Ax_n\|^2=\sum_{n=1}^{\infty}\|Ay_n\|^2=\sum_{n=1}^{\infty}\|A^*x_n\|^2$.

 In the situation of  operator-valued sequences, we set  up the   following definition.

 \begin{definition}\label{SCHMIDTGENERALIZED}
 Let $\theta :\mathcal{H}_0 \rightarrow \mathcal{H}_0$ be a conjugate-linear isometry. Let  $\{F_n\}_n$ be  an operator-valued orthonormal basis  in  $\mathcal{B}(\mathcal{H},\mathcal{H}_0)$	and $\{x_n\}_n$ be a  sequence in $\mathcal{H}_0$.  Define (the generalized Hilbert-Schmidt class) 
 \begin{align*}
 \mathcal{S}_{\theta,F,x}(\mathcal{H})\coloneqq \bigg\{ A\in \mathcal{B}(\mathcal{H}):& \sum_{n=1}^{\infty}\|AF_n^*x_n\|^2<\infty, \theta(F_mV^*A^*U^*F_n^*x_n)= F_nUAVF_m^*x_m,\\
 &\theta(F_mV^*AU^*F_n^*x_n)= F_nUA^*VF_m^*x_m,  \forall n, m \in \mathbb{N},\forall U,V \in \mathcal{B}(\mathcal{H}) \bigg\}.
 \end{align*}
 If $A \in \mathcal{S}_{\theta,F,x}(\mathcal{H})$, then we define
 \begin{align*}
 \sigma_{\theta, F,x}(A)\coloneqq\left(\sum_{n=1}^{\infty}\|AF_n^*x_n\|^2\right)^\frac{1}{2}.
 \end{align*}
 \end{definition}

 \begin{remark}
 Definition \ref{SCHMIDTGENERALIZED} reduces to the definition  of class of Hilbert-Schmidt operators (as well as Hilbert-Schmidt norm), Definition \ref{SCHMIDTCLASS}, given by R. Schatten and J. von Neumann,  whenever $\mathcal{H}_0=\mathbb{K}$, map $\theta$ is conjugation, $x_n=1, F_n:\mathcal{H}\ni h \mapsto \langle h, e_n\rangle \in \mathbb{K} , \forall n \in \mathbb{N}$, where $\{e_n\}_n$ is an orthonormal basis for $\mathcal{H}$. Then $\{F_n\}_n$ is an  operator-valued orthonormal basis  in  $\mathcal{B}(\mathcal{H},\mathbb{K})$. Observe now that the conditions $\overline{F_mV^*A^*U^*F_n^*x_n}= F_nUAVF_m^*x_m,$ $ \overline{F_mV^*AU^*F_n^*x_n}= F_nUA^*VF_m^*x_m,  \forall n, m \in \mathbb{N} $ holds for all $A, U, V\in \mathcal{B}(\mathcal{H})$. Indeed, for   $A, U, V\in \mathcal{B}(\mathcal{H})$, $\overline{F_mV^*A^*U^*F_n^*x_n}=\overline{F_mV^*A^*U^*F_n^*1}=\overline{F_mV^*A^*U^*e_n}=\overline{\langle V^*A^*U^*e_n , e_m \rangle } =\overline{\langle e_n , UAVe_m\rangle}=\langle UAVe_m ,e_n \rangle=F_nUAVe_m=F_nUAVF_m^*x_m, \forall n, m \in \mathbb{N} $. Similarly $ \overline{F_mV^*AU^*F_n^*x_n}= F_nUA^*VF_m^*x_m,  \forall n, m \in \mathbb{N},\forall U,V \in \mathcal{B}(\mathcal{H}) $. 
\end{remark}
\begin{theorem}\label{IDEAL}
Let $ A, B \in \mathcal{S}_{\theta,F,x}(\mathcal{H})$, $\alpha \in \mathbb{K}$,  $T \in \mathcal{B}(\mathcal{H})$. Then
\begin{enumerate}[\upshape(i)]
\item $ \alpha A\in \mathcal{S}_{\theta,F,x}(\mathcal{H})$ and $\sigma_{\theta,F,x}(\alpha A)=|\alpha|\sigma_{\theta,F,x}(A)$.
\item $ A+B\in \mathcal{S}_{\theta,F,x}(\mathcal{H})$ and $\sigma_{\theta, F,x}(A+B)\leq\sigma_{\theta,F,x}(A)+\sigma_{\theta,F,x}(B)$.
\item  $\mathcal{S}_{\theta,F,x}(\mathcal{H})$ is a subspace of   $\mathcal{B}(\mathcal{H})$. 
\item $ A^*\in \mathcal{S}_{\theta,F,x}(\mathcal{H})$ and $ \sigma_{\theta,F,x}(A^*)= \sigma_{\theta,F,x}(A)$.
\item  $ TA\in \mathcal{S}_{\theta,F,x}(\mathcal{H})$ and  $\sigma_{\theta,F,x}(TA)\leq \|T\|\sigma_{\theta,F,x}(A)$.
\item  $ AT\in \mathcal{S}_{\theta,F,x}(\mathcal{H})$ and  $\sigma_{\theta,F,x}(AT)\leq \|T\|\sigma_{\theta,F,x}(A)$.
\item If there exist $a>0$ and $2\leq p <\infty$  such that
\begin{align*}
a\|h\|&\leq \left(\sum_{n=1}^{\infty}|\langle h, F_n^*x_n\rangle |^p\right)^\frac{1}{p},~ \forall h \in \mathcal{H}, 
\end{align*}		
 then $\|A\|\leq \sigma_{\theta,F,x}(A)/a, \forall A \in \mathcal{S}_{\theta,F,x}(\mathcal{H})$, $\sigma_{\theta,F,x}(\cdot)$ is a norm on      $\mathcal{S}_{\theta,F,x}(\mathcal{H})$  and  $\mathcal{S}_{\theta,F,x}(\mathcal{H})$ is complete in this norm.  In particular, if $\{F_n^*x_n\}_n $ is a frame for $\mathcal{H}$, then $\mathcal{S}_{\theta,F,x}(\mathcal{H})$ is complete.
 \item If $A^*A\leq B^*B$, then $\sigma_{\theta,F,x}(A)\leq \sigma_{\theta,F,x}(B)$.
 \item If $C \in \mathcal{B}(\mathcal{H})$, then $ C\in \mathcal{S}_{\theta,F,x}(\mathcal{H})$ if and only if $ [C] \in \mathcal{S}_{\theta,F,x}(\mathcal{H})$. In this case, $\sigma_{\theta,F,x}(C)=\sigma_{\theta,F,x}([C])$.   
 \item If $\|A\|<1$, then $\sigma_{\theta,F,x}(A^n)\rightarrow 0$ as $n\rightarrow \infty$.
 \item If $F_n^*x_n$ is an eigenvector for $A$ with eigenvalue $\lambda_n$ for each $n\in \mathbb{N}$, then $\sigma_{\theta,F,x}(A)^2=\sum_{n=1}^{\infty}|\lambda_n|^2\|x_n\|^2.$
\end{enumerate}
\end{theorem} 
\begin{proof}
\begin{enumerate}[\upshape(i)]
\item $\sigma_{\theta, F,x}(\alpha A)^2=\sum_{n=1}^{\infty}\|\alpha AF_n^*x_n\|^2=|\alpha|^2\sum_{n=1}^{\infty}\|AF_n^*x_n\|^2=|\alpha|^2\sigma_{\theta,F,x}(A)^2$, $\theta(F_mV^*(\alpha A)^*U^*F_n^*x_n)$ $=\alpha \theta(F_mV^* A^*U^*F_n^*x_n)= \alpha F_nUAVF_m^*x_m=F_nU(\alpha A)VF_m^*x_m$, $\forall n,m \in \mathbb{N}, \forall U,V \in \mathcal{B}(\mathcal{H})$. Similarly $\theta(F_mV^*(\alpha A)U^*F_n^*x_n)=F_nU(\alpha A)^*VF_m^*x_m$, $\forall n,m \in \mathbb{N}, \forall U,V \in \mathcal{B}(\mathcal{H})$.
\item $\sigma_{\theta,F,x}(A+B)=\left(\sum_{n=1}^{\infty}\|AF_n^*x_n+BF_n^*x_n\|^2\right)^{1/2}\leq \left(\sum_{n=1}^{\infty}\|AF_n^*x_n\|^2\right)^{1/2}+\left(\sum_{n=1}^{\infty}\|BF_n^*x_n\|^2\right)^{1/2}=\sigma_{\theta,F,x}(A)+\sigma_{\theta,F,x}(B)$, $\theta(F_mV^*(A+B)^*U^*F_n^*x_n)=\theta(F_mV^*A^*U^*F_n^*x_n)+\theta(F_mV^*B^*U^*F_n^*x_n)=F_nUAVF_m^*x_m+F_nUBVF_m^*x_m=F_nU(A+B)VF_m^*x_m$, and $\theta(F_mV^*(A+B)U^*F_n^*x_n)=F_nU(A+B)^*VF_m^*x_m$, $\forall n,m \in \mathbb{N}, \forall U,V \in \mathcal{B}(\mathcal{H})$.
\item comes from (i) and (ii).
\item For every $ k \in \mathbb{N}$, 
\begin{align*}
\sum_{n=1}^{k}\|A^*F_n^*x_n\|^2&=\sum_{n=1}^{k}\sum_{m=1}^{\infty}\|F_mA^*F_n^*x_n\|^2=\sum_{n=1}^{k}\sum_{m=1}^{\infty}\|\theta(F_mA^*F_n^*x_n)\|^2\\
&=\sum_{n=1}^{k}\sum_{m=1}^{\infty}\|F_nAF_m^*x_m\|^2
=\sum_{m=1}^{\infty}\sum_{n=1}^{k}\|F_nAF_m^*x_m\|^2\\
&\leq \sum_{m=1}^{\infty}\sum_{n=1}^{\infty}\|F_n(AF_m^*x_m)\|^2=\sum_{m=1}^{\infty}\|AF_m^*x_m\|^2.
\end{align*}
Therefore $\sum_{n=1}^{\infty}\|A^*F_n^*x_n\|^2<\infty$. A similar procedure gives $	\sigma_{\theta,F,x}(A^*)^2=\sum_{n=1}^{\infty}\|A^*F_n^*x_n\|^2=\sum_{m=1}^{\infty}\sum_{n=1}^{\infty}\|F_n(AF_m^*x_m)\|^2=\sum_{m=1}^{\infty}\|AF_m^*x_m\|^2=\sigma_{\theta,F,x}(A)^2.$
\item 	$\sigma_{\theta,F,x}(TA)^2=\sum_{n=1}^{\infty}\|TAF_n^*x_n\|^2\leq\|T\|^2\sum_{n=1}^{\infty}\|AF_n^*x_n\|^2=\|T\|^2\sigma_{\theta,F,x}(A)^2$, $\theta(F_mV^*(TA)^*U^*F_n^*x_n)$ $=\theta(F_mV^*A^*(T^*U^*)F_n^*x_n)= F_n(T^*U^*)^*AVF_m^*x_m=F_nU(TA)VF_m^*x_m $ and $\theta(F_mV^*(TA)U^*F_n^*x_n)$ $=F_nU(TA)^*VF_m^*x_m $, $\forall n,m \in \mathbb{N}, \forall U,V \in \mathcal{B}(\mathcal{H})$.
\item It is enough to show $(AT)^* \in \mathcal{S}_{\theta, F,x}(\mathcal{H})$ (using (iv)). For,  $\sigma_{\theta,F,x}((AT)^*)^2=\sum_{n=1}^{\infty}\|T^*A^*F_n^*x_n\|^2\leq\|T^*\|^2\sum_{n=1}^{\infty}\|A^*F_n^*x_n\|^2=\|T\|^2\sigma_{F,x}(A^*)^2=\|T\|^2\sigma_{\theta,F,x}(A)^2$. Further,  $\theta(F_mV^*(AT)^*U^*F_n^*x_n)$ $=\theta(F_m(V^*T^*)A^*U^*F_n^*x_n)= F_nUA(V^*T^*)^*F_m^*x_m=F_nU(AT)VF_m^*x_m $, $\theta(F_mV^*(AT)U^*F_n^*x_n)=F_nU(AT)^*VF_m^*x_m $, $\forall n,m \in \mathbb{N}, \forall U,V \in \mathcal{B}(\mathcal{H})$.
\item For all $h \in \mathcal{H}$,

\begin{align*}
a\|A^*h\|&\leq \left(\sum_{n=1}^{\infty}|\langle A^*h, F_n^*x_n\rangle |^p\right)^\frac{1}{p}=\left(\sum_{n=1}^{\infty}|\langle h, AF_n^*x_n\rangle |^p\right)^\frac{1}{p}\\
&\leq \|h\|\left(\sum_{n=1}^{\infty} \|AF_n^*x_n \|^p\right)^\frac{1}{p}\leq \|h\|\left(\sum_{n=1}^{\infty} \|AF_n^*x_n \|^2\right)^\frac{1}{2}=\|h\|\sigma_{\theta,F,x}(A).
\end{align*}
This gives $a\|A^*\|=a\|A\|\leq\sigma_{\theta,F,x}(A)$ which tells $\sigma_{\theta,F,x}(A)=0$ $\Rightarrow$ $A=0$. 
Let $\{A_n\}_{n=1}^{\infty}$ be a Cauchy sequence in  $\mathcal{S}_{\theta, F,x}(\mathcal{H})$. Then $\|A_n-A_m\|\leq \sigma_{\theta,F,x}(A_n-A_m)/a, \forall n,m\in \mathbb{N}$ tells that $\{A_n\}_{n=1}^{\infty}$ is a Cauchy sequence in $\mathcal{B}(\mathcal{H})$ w.r.t. operator-norm. Let $A\coloneqq \lim_{n \rightarrow \infty}A_n$, in the operator-norm. Choose $N \in \mathbb{N}$ such that $\sum_{k=1}^{\infty}\|(A_n-A_m)F_k^*x_k\|^2=\sigma_{\theta,F,x}(A_n-A_m)<1,\forall n,m \geq N$. This gives that for each $r\in \mathbb{N}$, $\sum_{k=1}^{r}\|(A_n-A_m)F_k^*x_k\|^2<1,\forall n,m \geq N$ $\Rightarrow$ for each $r\in \mathbb{N}$, $\sum_{k=1}^{r}\|(A_n-A)F_k^*x_k\|^2=\lim_{m\rightarrow \infty}\sum_{k=1}^{r}\|(A_n-A_m)F_k^*x_k\|^2\leq1,\forall n \geq N$ $\Rightarrow$ $\sum_{k=1}^{\infty}\|(A_n-A)F_k^*x_k\|^2\leq1, \forall n \geq N$.

Now for $n, m \in \mathbb{N}$ and $U,V \in \mathcal{B}(\mathcal{H})$, $\theta(F_mV^*(A-A_N)^*U^*F_n^*x_n)=\theta(F_mV^*\lim_{k \rightarrow \infty}(A_k^*-A_N^*)U^*F_n^*x_n)=\lim_{k \rightarrow \infty}\theta(F_mV^*(A_k^*-A_N^*)U^*F_n^*x_n)=\lim_{k \rightarrow \infty}(\theta(F_mV^*A_k^*U^*F_n^*x_n)-\theta(F_mV^*A_N^*U^*F_n^*x_n))$ $=\lim_{k \rightarrow \infty}(F_nUA_kVF_m^*x_m-F_nUA_NVF_m^*x_m)=F_nU\lim_{k \rightarrow \infty}(A_k-A_N)VF_m^*x_m= F_nU(A-A_N)VF_m^*x_m$. Similarly  $\theta(F_mV^*(A-A_N)U^*F_n^*x_n)= F_nU(A-A_N)^*VF_m^*x_m$. Therefore $A-A_N\in \mathcal{S}_{\theta, F,x}(\mathcal{H})$. Hence  $A=(A-A_N)+A_N\in \mathcal{S}_{\theta, F,x}(\mathcal{H})$.
\item $\sigma_{\theta,F,x}(A)^2=\sum_{n=1}^{\infty}\|AF_n^*x_n\|^2=\sum_{n=1}^{\infty}\langle A^*AF_n^*x_n, F_n^*x_n\rangle\leq\sum_{n=1}^{\infty}\langle B^*BF_n^*x_n,F_n^* x_n\rangle=\sum_{n=1}^{\infty}\|BF_n^*x_n\|^2$ $=\sigma_{\theta,F,x}(B)^2  $.
\item Let $C=W[C]$ be the polar decomposition of $C$ as in Theorem \ref{POLARDECOMPOSITIONSCHATTEN}. $(\Rightarrow)$ From (ii) in Theorem \ref{POLARDECOMPOSITIONSCHATTEN} we have  $[C]=W^*C$. Now (v) tells that $[C]=W^*C \in  \mathcal{S}_{\theta,F,x}(\mathcal{H})$. $(\Leftarrow)$ Polar decomposition of $C$ and (v) give $C \in  \mathcal{S}_{\theta,F,x}(\mathcal{H})$. Further, using (v), $\sigma_{\theta,F,x}(C)=\sigma_{\theta,F,x}(W[C])\leq\|W\|\sigma_{\theta,F,x}([C])=\sigma_{\theta,F,x}([C])=\sigma_{\theta,F,x}(W^*C)\leq \|W^*\|\sigma_{\theta,F,x}(C)=\sigma_{\theta,F,x}(C)$.
\item $0 \leq\sigma_{\theta,F,x}(A^n)\leq \|A^{n-1}\|\sigma_{\theta,F,x}(A)\leq \|A\|^{n-1}\sigma_{\theta,F,x}(A)\rightarrow 0$ as $n\rightarrow \infty$.
\item $\sigma_{\theta,F,x}(A)^2=\sum_{n=1}^{\infty}|\lambda_n|^2\|F_n^*x_n\|^2=\sum_{n=1}^{\infty}|\lambda_n|^2\|x_n\|^2.$
\end{enumerate}
\end{proof}
Theorem \ref{IDEAL} says that $\mathcal{S}_{\theta,F,x}(\mathcal{H})$ is a two sided  ideal in $\mathcal{B}(\mathcal{H})$. Further, if assumption in (vii) holds, then it is two sided  closed ideal.

 \begin{lemma}\label{TRACECLASSEXISTENCELEMMA}
 If $ A, B \in \mathcal{S}_{\theta,F,x}(\mathcal{H})$, then  the series $\sum_{n=1}^{\infty}|\langle AF_n^*x_n, BF_n^*x_n\rangle |$ converges. Further, if $\mathcal{H}$ and  $\mathcal{H}_0 $ are over $ \mathbb{C}$, then 
 \begin{align*}
 4\sum_{n=1}^{\infty}\langle AF_n^*x_n, BF_n^*x_n\rangle =& \sigma_{\theta,F,x}(A+B)^2-\sigma_{\theta,F,x}(A-B)^2+i\sigma_{\theta,F,x}(A+iB)^2-i\sigma_{\theta,F,x}(A-iB)^2 
 \end{align*}
  and  if $\mathcal{H}$ and  $\mathcal{H}_0 $ are over $ \mathbb{R}$, then
 \begin{align*}
  4\sum_{n=1}^{\infty}\langle AF_n^*x_n, BF_n^*x_n\rangle =&\sigma_{\theta,F,x}(A+B)^2-\sigma_{\theta,F,x}(A-B)^2.
 \end{align*}
 \end{lemma}
 \begin{proof}
 Let $	\mathcal{H} $ and $ \mathcal{H}_0 $ be over $ \mathbb{C}$. For all $m \in \mathbb{N}$, we have  $\sum_{n=1}^{m}|\langle AF_n^*x_n, BF_n^*x_n\rangle |\leq (\sum_{n=1}^{m}\|AF_n^*x_n\|^2)^{1/2}$ $(\sum_{n=1}^{m}\|BF_n^*x_n\|^2)^{1/2}\leq \sigma_{\theta, F,x}(A)\sigma_{\theta, F,x}(B)$. We next use the polarization identity,
 
 \begin{align*}
 &4\sum_{n=1}^{\infty}\langle AF_n^*x_n, BF_n^*x_n\rangle =\\
 & \sum_{n=1}^{\infty}(\|AF_n^*x_n+BF_n^*x_n\|^2-\|AF_n^*x_n-BF_n^*x_n\|^2+i\|AF_n^*x_n+iBF_n^*x_n\|^2-i\|AF_n^*x_n-iBF_n^*x_n\|^2)\\
 &=\sum_{n=1}^{\infty}\|(A+B)F_n^*x_n\|^2-\sum_{n=1}^{\infty}\|(A-B)F_n^*x_n\|^2+i\sum_{n=1}^{\infty}\|(A+iB)F_n^*x_n\|^2-i\sum_{n=1}^{\infty}\|(A-iB)F_n^*x_n\|^2\\
 &=\sigma_{\theta,F,x}(A+B)^2-\sigma_{\theta,F,x}(A-B)^2+i\sigma_{\theta,F,x}(A+iB)^2-i\sigma_{\theta,F,x}(A-iB)^2.
 \end{align*}	
 The argument is similar if $	\mathcal{H} $ and $ \mathcal{H}_0 $ are over $ \mathbb{R}$.
 \end{proof}
 \begin{definition}
 Given $ A, B \in \mathcal{S}_{\theta, F,x}(\mathcal{H})$, we define 
 \begin{align}\label{IPWELL}
 \langle A, B\rangle \coloneqq \sum_{n=1}^{\infty}\langle AF_n^*x_n, BF_n^*x_n\rangle.
 \end{align}	
 \end{definition}\label{HILBERTSCHMIDTHHILBERT}
 Lemma \ref{TRACECLASSEXISTENCELEMMA} says that the series in Equation (\ref{IPWELL}) is well-defined. We observe that $\sigma_{\theta, F,x}(A)^2= \langle A, A\rangle,  \forall A \in \mathcal{S}_{\theta,F,x}(\mathcal{H})$.
 \begin{proposition}\label{SHIFTPROPOSITION}
 \begin{enumerate}[\upshape(i)]
 \item $\langle A, A\rangle\geq0 ,\forall A \in \mathcal{S}_{\theta,F,x}(\mathcal{H})$.
 \item If there exist $a>0$ and $2\leq p <\infty$  such that
 \begin{align*}
 a\|h\|&\leq \left(\sum_{n=1}^{\infty}|\langle h, F_n^*x_n\rangle |^p\right)^\frac{1}{p},~ \forall h \in \mathcal{H}, 
 \end{align*}		
 then $\langle A, A\rangle=0$ implies that  $A=0.$
 \item $\langle A+B, C \rangle =\langle A, C \rangle+\langle B, C \rangle $, $\langle \alpha A, B\rangle=\alpha \langle  A, B \rangle, \forall A,B,C \in \mathcal{S}_{\theta,F,x}(\mathcal{H}), \forall \alpha \in \mathbb{K}$.
 \item $\overline{\langle A, B\rangle}=\langle B, A\rangle,  \forall A,B \in \mathcal{S}_{\theta,F,x}(\mathcal{H})$.
 \item $\langle A^*, B^*\rangle=\overline{\langle A, B\rangle}, \forall A,B \in \mathcal{S}_{\theta,F,x}(\mathcal{H})$.
 \item  If $ A, B \in \mathcal{S}_{\theta,F,x}(\mathcal{H})$ and $T \in \mathcal{B}(\mathcal{H})$, then $ \langle TA, B\rangle= \langle A, T^*B\rangle$ and $  \langle AT, B\rangle= \langle A, BT^*\rangle$.
 \item $ |\langle A, B \rangle| \leq \sigma_{\theta,F,x}(A)\sigma_{\theta,F,x}(B),  \forall A,B \in \mathcal{S}_{\theta,F,x}(\mathcal{H})$.
 \end{enumerate}
 \end{proposition}
 \begin{proof}
 \begin{enumerate}[\upshape(i)]
\item $\langle A, A\rangle=\sum_{n=1}^{\infty}\| AF_n^*x_n\|^2\geq0$.
\item From (vii) in Theorem \ref{IDEAL}, $0=\langle A, A\rangle=\sigma_{\theta,F,x}(A)\geq a\|A\|$.
\item $\langle A+B, C \rangle =\sum_{n=1}^{\infty}\langle AF_n^*x_n, CF_n^*x_n\rangle+\sum_{n=1}^{\infty}\langle BF_n^*x_n, CF_n^*x_n\rangle =\langle A, C \rangle+\langle B, C \rangle $, $\langle \alpha A, B\rangle=\alpha\sum_{n=1}^{\infty}\langle AF_n^*x_n, BF_n^*x_n\rangle=\alpha \langle  A, B \rangle$.
\item $\overline{\langle A, B\rangle}=\sum_{n=1}^{\infty}\langle BF_n^*x_n, AF_n^*x_n\rangle=\langle B, A\rangle$.
\item We consider the case $\mathbb{K}=\mathbb{C}$, the case $\mathbb{K}=\mathbb{R}$ is similar. For $A,B \in \mathcal{S}_{\theta, F,x}(\mathcal{H})$, by taking the help of Lemma \ref{TRACECLASSEXISTENCELEMMA}
 \begin{align*}
&4\langle A^*, B^*\rangle=4\sum_{n=1}^{\infty}\langle A^*F_n^*x_n, B^*F_n^*x_n\rangle\\
&=\sigma_{\theta,F,x}(A^*+B^*)^2-\sigma_{\theta,F,x}(A^*-B^*)^2+i\sigma_{\theta,F,x}(A^*+iB^*)^2-i\sigma_{\theta,F,x}(A^*-iB^*)^2\\
&=\sigma_{\theta,F,x}((A+B)^*)^2-\sigma_{\theta,F,x}((A-B)^*)^2+i\sigma_{\theta,F,x}((A-iB)^*)^2-i\sigma_{\theta,F,x}((A+iB)^*)^2\\
&=\sigma_{\theta,F,x}(A+B)^2-\sigma_{\theta,F,x}(A-B)^2+i\sigma_{\theta,F,x}(A-iB)^2-i\sigma_{\theta,F,x}(A+iB)^2\\
 &=\overline{\sigma_{\theta,F,x}(A+B)^2-\sigma_{\theta,F,x}(A-B)^2+i\sigma_{\theta,F,x}(A+iB)^2-i\sigma_{\theta,F,x}(A-iB)^2}\\
 &=4\overline{\sum_{n=1}^{\infty}\langle AF_n^*x_n, BF_n^*x_n\rangle}=4\overline{\langle A, B\rangle}.
 \end{align*}
 \item $  \langle TA, B\rangle=\sum_{n=1}^{\infty}\langle TAF_n^*x_n, BF_n^*x_n\rangle=\sum_{n=1}^{\infty}\langle AF_n^*x_n, T^*BF_n^*x_n\rangle=\langle A, T^*B\rangle,$ $  \langle AT, B\rangle=\overline{\langle B, AT\rangle }=\langle B^*, (AT)^*\rangle=\overline{\langle  T^*A^*, B^*\rangle}=\overline{\langle  A^*, TB^*\rangle}=\overline{\langle  A^*, (BT^*)^*\rangle}= \langle A, BT^*\rangle$.
 \item $ |\langle A, B \rangle|\leq \sum_{n=1}^{\infty}\|AF_n^*x_n\|\|BF_n^*x_n\|\leq (\sum_{n=1}^{\infty}\|AF_n^*x_n\|^2)^{1/2}(\sum_{n=1}^{\infty}\|BF_n^*x_n\|^2)^{1/2}=  \sigma_{\theta,F,x}(A)\sigma_{\theta,F,x}(B)$.
 \end{enumerate}
 \end{proof}

\section{Generalized trace class}\label{TRACECLASSSECTION}
 \begin{lemma}\label{GTRACECLASSEXISTENCELEMMA}
 If $A=BC$, where $B,C \in \mathcal{S}_{\theta, F,x}(\mathcal{H})$, then  the series $\sum_{n=1}^{\infty}|\langle AF_n^*x_n,F_n^*x_n \rangle |$ converges.	
 \end{lemma}
 \begin{proof}
 Since $B^*\in \mathcal{S}_{\theta,F,x}(\mathcal{H})$ and  $\langle AF_n^*x_n,F_n^*x_n \rangle=\langle CF_n^*x_n,B^*F_n^*x_n \rangle, \forall n \in \mathbb{N}$, the convergence follows from Lemma \ref{TRACECLASSEXISTENCELEMMA}.
 \end{proof}
 \begin{definition}\label{TRACECLASSDEFINITION}
 We define the generalized trace class as 
 \begin{align*}
 \mathcal{T}_{\theta,F,x}(\mathcal{H})\coloneqq\{AB: A,B \in \mathcal{S}_{\theta,F,x}(\mathcal{H})\}.
 \end{align*}
 If $A \in \mathcal{T}_{\theta,F,x}(\mathcal{H})$, then we define the trace of $A$ as 
 
 \begin{align*}
 \operatorname{Tr}_{\theta,F,x}(A)\coloneqq\sum_{n=1}^{\infty}\langle AF_n^*x_n, F_n^*x_n\rangle.
 \end{align*}
 \end{definition} 	
Since $\mathcal{S}_{\theta,F,x}(\mathcal{H})$ is closed under multiplication, we naturally have $\mathcal{T}_{\theta,F,x}(\mathcal{H})\subseteq\mathcal{S}_{\theta,F,x}(\mathcal{H})$.

\begin{lemma}\label{CHARACTERIZATIONST}
For $ A \in\mathcal{B}(\mathcal{H})$,  we have \text{\upshape(i)} $\Rightarrow$  \text{\upshape(ii)} $\Rightarrow$ \text{\upshape(iii)} $\Rightarrow$ \text{\upshape(iv)}, where
 \begin{enumerate}[\upshape(i)]
\item $[A]^{1/2} \in \mathcal{S}_{\theta,F,x}(\mathcal{H})$.
\item $A \in \mathcal{T}_{\theta,F,x}(\mathcal{H})$.
\item $[A] \in \mathcal{T}_{\theta,F,x}(\mathcal{H})$.
\item $\sum_{n=1}^{\infty}\langle [A]F_n^*x_n,F_n^*x_n \rangle $ converges. 
\end{enumerate}
If $\theta(F_mV^*[A]^{1/2}U^*F_n^*x_n)= F_nU[A]^{1/2}VF_m^*x_m$,  $\forall n, m \in \mathbb{N}, \forall U,V \in \mathcal{B}(\mathcal{H})$ -----$(\star)$, then \text{\upshape(iv)} $\Rightarrow$ \text{\upshape(i)}.
 \end{lemma}
 \begin{proof}
 Let $A=W[A]$ be the polar decomposition of $A$ as given in Theorem \ref{POLARDECOMPOSITIONSCHATTEN}. Then $[A]=W^*A$.
 \begin{enumerate}[\upshape(i)]
 \item $\Rightarrow$ (ii) Since $\mathcal{S}_{\theta,F,x}(\mathcal{H})$ is an ideal in $\mathcal{B}(\mathcal{H})$ and $[A]^{1/2} \in \mathcal{S}_{\theta,F,x}(\mathcal{H})$ we have  $W^*[A]^{1/2} \in \mathcal{S}_{\theta,F,x}(\mathcal{H})$. But then $A=(W^*[A]^{1/2})[A]^{1/2} \in \mathcal{T}_{\theta,F,x}(\mathcal{H})$, from the Definition \ref{TRACECLASSDEFINITION}.
 \item $\Rightarrow$ (iii) From Definition \ref{TRACECLASSDEFINITION} $A=BC$ for some $B,C \in \mathcal{S}_{\theta,F,x}(\mathcal{H})$. Then $W^*B \in \mathcal{S}_{\theta,F,x}(\mathcal{H})$ which gives $[A]=W^*A=(W^*B)C \in \mathcal{T}_{\theta,F,x}(\mathcal{H})$.
 \item $\Rightarrow$ (iv) Let $A=BC$ for some $B,C \in \mathcal{S}_{\theta,F,x}(\mathcal{H})$. From Lemma \ref{GTRACECLASSEXISTENCELEMMA}, $\sum_{n=1}^{\infty}\langle [A]F_n^*x_n,F_n^*x_n \rangle $ converges.
 \item and $(\star)$ $\Rightarrow$ (i)  $\sum_{n=1}^{\infty}\|[A]^{1/2}F_n^*x_n\|^2=\sum_{n=1}^{\infty}\langle [A]F_n^*x_n,F_n^*x_n \rangle<\infty $.
 \end{enumerate}
 \end{proof}
\begin{theorem}\label{TRACEIPCONNECTION}
Let  $A, B  \in \mathcal{T}_{\theta,F,x}(\mathcal{H})$, $\alpha \in \mathbb{K}$ and $T \in \mathcal{B}(\mathcal{H})$. Then 
\begin{enumerate}[\upshape(i)]
\item  $A^*  \in \mathcal{T}_{\theta,F,x}(\mathcal{H})$ and $ \operatorname{Tr}_{\theta,F,x}(A^*)= \overline{\operatorname{Tr}_{\theta,F,x}(A)}$.
\item $\alpha A \in \mathcal{T}_{\theta,F,x}(\mathcal{H})$ and $ \operatorname{Tr}_{\theta,F,x}(\alpha A)=\alpha \operatorname{Tr}_{\theta,F,x}(A)$.
\item  $TA  \in \mathcal{T}_{\theta,F,x}(\mathcal{H})$, $AT \in \mathcal{T}_{\theta,F,x}(\mathcal{H})$.
\item  $\operatorname{Tr}_{\theta,F,x}(A^*A)=\sigma_{\theta, F,x}(A)^2=\operatorname{Tr}_{\theta,F,x}(AA^*)$.
\item  If $\theta(F_mV^*[A+B]^{1/2}U^*F_n^*x_n)= F_nU[A+B]^{1/2}VF_m^*x_m$, $\forall n, m \in \mathbb{N}, \forall U,V \in \mathcal{B}(\mathcal{H})$, then $ A +B \in \mathcal{T}_{\theta,F,x}(\mathcal{H})$ and $ \operatorname{Tr}_{\theta,F,x}(A+B)= \operatorname{Tr}_{\theta,F,x}(A)+\operatorname{Tr}_{\theta,F,x}(B)$.
\item $\operatorname{Tr}_{\theta,F,x}(B^*A)=\langle A, B\rangle.$
\item $|\operatorname{Tr}_{\theta,F,x}(B^*A)|\leq (\operatorname{Tr}_{\theta,F,x}(A^*A))^{1/2}(\operatorname{Tr}_{\theta,F,x}(B^*B))^{1/2}.$
\item $|\operatorname{Tr}_{\theta,F,x}(A^2)|\leq\operatorname{Tr}_{\theta,F,x}(A^*A).$
\item If $0\leq A\leq B$, then $\operatorname{Tr}_{\theta,F,x}(A)\leq \operatorname{Tr}_{\theta,F,x}(B)$.
\item If $F_n^*x_n$ is an eigenvector for $A$ with eigenvalue $\lambda_n$ for each $n\in \mathbb{N}$, then $\operatorname{Tr}_{\theta,F,x}(A)=\sum_{n=1}^{\infty}\lambda_n\|x_n\|^2.$
\item If $A\geq0$, $F_n^*x_n$ is an eigenvector for $A$ with eigenvalue $\lambda_n$ for each $n\in \mathbb{N}$ and $\|x_n\|\geq1,\forall n \in \mathbb{N}$, then $\sigma_{\theta,F,x}(A)\leq\operatorname{Tr}_{\theta,F,x}(A).$
\end{enumerate}
\end{theorem}
\begin{proof}
Let $A=CD$ for some $C,D \in \mathcal{S}_{\theta,F,x}(\mathcal{H})$.	
\begin{enumerate}[\upshape(i)]
\item  Since $ \mathcal{S}_{\theta,F,x}(\mathcal{H})$ is closed under $*$ and $A^*=D^*C^*$ we see that $A^*  \in \mathcal{T}_{\theta,F,x}(\mathcal{H})$ and $ \operatorname{Tr}_{\theta,F,x}(A^*)=\sum_{n=1}^{\infty}\langle A^*F_n^*x_n, F_n^*x_n\rangle =\overline{\sum_{n=1}^{\infty}\langle AF_n^*x_n, F_n^*x_n\rangle} =\overline{\operatorname{Tr}_{\theta,F,x}(A)}$.
\item  Since  $ \alpha C\in \mathcal{S}_{\theta,F,x}(\mathcal{H})$, we have $\alpha A \in \mathcal{T}_{\theta,F,x}(\mathcal{H})$ and $ \operatorname{Tr}_{\theta,F,x}(\alpha A)=\sum_{n=1}^{\infty}\langle \alpha AF_n^*x_n, F_n^*x_n\rangle=\alpha\sum_{n=1}^{\infty}\langle AF_n^*x_n, F_n^*x_n\rangle=\alpha \operatorname{Tr}_{\theta,F,x}(A)$.
\item $TA=(TC)D \in \mathcal{T}_{\theta,F,x}(\mathcal{H})$, $AT=C(DT) \in \mathcal{T}_{\theta,F,x}(\mathcal{H})$. 
\item  $\operatorname{Tr}_{\theta,F,x}(A^*A)=\sum_{n=1}^{\infty}\langle A^*AF_n^*x_n, F_n^*x_n\rangle=\sum_{n=1}^{\infty}\| AF_n^*x_n\|^2=\sigma_{\theta, F,x}(A)^2=\sigma_{\theta, F,x}(A^*)^2=\operatorname{Tr}_{\theta,F,x}(AA^*)$.
\item  Let $A+B=W^*[A+B]$ be the polar decomposition of $A+B$, given by Theorem \ref{POLARDECOMPOSITIONSCHATTEN}. From (iii) we see that $W^*A,W^*B\in \mathcal{T}_{\theta,F,x}(\mathcal{H})$. Then from Lemma \ref{GTRACECLASSEXISTENCELEMMA}, both series $\sum_{n=1}^{\infty}\langle W^*AF_n^*x_n, F_n^*x_n\rangle$, $\sum_{n=1}^{\infty}\langle W^*BF_n^*x_n, F_n^*x_n\rangle$ converge which implies $\sum_{n=1}^{\infty}\langle [A+B]F_n^*x_n, F_n^*x_n\rangle=\sum_{n=1}^{\infty}\langle (W^*A+W^*B)F_n^*x_n, F_n^*x_n\rangle<\infty$. Lemma \ref{CHARACTERIZATIONST} now tells that $ A +B \in \mathcal{T}_{\theta,F,x}(\mathcal{H})$. We now easily see $ \operatorname{Tr}_{\theta,F,x}(A+B)=\sum_{n=1}^{\infty}\langle (A+B)F_n^*x_n, F_n^*x_n\rangle= \operatorname{Tr}_{\theta,F,x}(A)+\operatorname{Tr}_{\theta,F,x}(B)$. 
\item $\operatorname{Tr}_{\theta,F,x}(B^*A)=\sum_{n=1}^{\infty}\langle B^*AF_n^*x_n, F_n^*x_n\rangle=\sum_{n=1}^{\infty}\langle AF_n^*x_n, BF_n^*x_n\rangle=\langle A, B\rangle$.
\item $|\operatorname{Tr}_{\theta,F,x}(B^*A)|^2=|\langle A, B\rangle|^2\leq \sigma_{\theta, F,x}(A)^2\sigma_{\theta, F,x}(B)^2= \operatorname{Tr}_{\theta,F,x}(A^*A)\operatorname{Tr}_{\theta,F,x}(B^*B).$
\item $|\operatorname{Tr}_{\theta,F,x}(A^2)|=|\langle A,A^* \rangle|\leq \sigma_{\theta, F,x}(A)\sigma_{\theta, F,x}(A^*)=\sigma_{\theta, F,x}(A)^2=\operatorname{Tr}_{\theta,F,x}(A^*A).$
\item $\operatorname{Tr}_{\theta,F,x}(A)=\sum_{n=1}^{\infty}\langle AF_n^*x_n, F_n^*x_n\rangle\leq \sum_{n=1}^{\infty}\langle BF_n^*x_n, F_n^*x_n\rangle= \operatorname{Tr}_{\theta,F,x}(B)$.
\item $\operatorname{Tr}_{\theta,F,x}(A)=\sum_{n=1}^{\infty} \lambda_n \|F_n^*x_n\|^2 =\sum_{n=1}^{\infty}\lambda_n\|x_n\|^2.$
\item $\sigma_{\theta,F,x}(A)=(\sum_{n=1}^{\infty}\|\lambda_nF_n^*x_n\|^2)^{1/2}=(\sum_{n=1}^{\infty}\|\lambda_nx_n\|^2)^{1/2}\leq\sum_{n=1}^{\infty}\|\lambda_nx_n\|\leq \sum_{n=1}^{\infty}|\lambda_n|\|x_n\|^2=\sum_{n=1}^{\infty}\lambda_n\|x_n\|^2 =\operatorname{Tr}_{\theta,F,x}(A).$
\end{enumerate} 	
\end{proof}
\begin{corollary}\label{TRACIALCOROLLARY}
 If $A  \in \mathcal{T}_{\theta,F,x}(\mathcal{H})$ and $T \in \mathcal{B}(\mathcal{H})$, then $ \operatorname{Tr}_{\theta,F,x}(TA)= \operatorname{Tr}_{\theta,F,x}(AT)$.	
 
\end{corollary}
\begin{proof}
We shall write $A=CD$ for some $C,D \in \mathcal{S}_{\theta,F,x}(\mathcal{H})$. Using  (v) in Theorem \ref{TRACEIPCONNECTION} and (v) in Proposition \ref{SHIFTPROPOSITION}, $ \operatorname{Tr}_{\theta,F,x}(TA)=\operatorname{Tr}_{\theta,F,x}((TC)D)=\langle D, C^*T^*\rangle =\langle DT,C^*\rangle=\operatorname{Tr}_{\theta,F,x}(CDT)= \operatorname{Tr}_{\theta,F,x}(AT)$.	
\end{proof}
\begin{definition}
For  $A  \in \mathcal{T}_{\theta,F,x}(\mathcal{H})$, we define 
\begin{align*}
\tau_{\theta,F,x}(A)\coloneqq \operatorname{Tr}_{\theta,F,x}([A]).
\end{align*}
\end{definition}
Lemma \ref{CHARACTERIZATIONST} says that $A  \in \mathcal{T}_{\theta,F,x}(\mathcal{H})$ implies $[A]  \in \mathcal{T}_{\theta,F,x}(\mathcal{H})$. Therefore $\tau_{\theta,F,x}(\cdot)$ is well-defined.

\begin{proposition}\label{PROPOSITIONTRACE}
 Let $A \in \mathcal{B}(\mathcal{H})$. 
 \begin{enumerate}[\upshape(i)]
 \item If $A  \in \mathcal{T}_{\theta,F,x}(\mathcal{H})$, then  $[A] \in \mathcal{T}_{\theta,F,x}(\mathcal{H})$ and  $\tau_{\theta,F,x}(A)=\tau_{\theta,F,x}([A])$.
 \item If $[A]  \in \mathcal{T}_{\theta,F,x}(\mathcal{H})$ and $\theta(F_mV^*[A]^{1/2}U^*F_n^*x_n)= F_nU[A]^{1/2}VF_m^*x_m$, $ \forall n, m \in \mathbb{N}, \forall U,V \in \mathcal{B}(\mathcal{H})$, then  $A \in \mathcal{T}_{\theta,F,x}(\mathcal{H})$ and  $\tau_{\theta,F,x}([A])=\tau_{\theta,F,x}(A)$.
 \item If $[A]^{1/2} \in \mathcal{S}_{\theta,F,x}(\mathcal{H})$, then $A \in \mathcal{T}_{\theta,F,x}(\mathcal{H})$ and $\sigma_{\theta,F,x}([A]^{1/2})^2=\tau_{\theta,F,x}(A)$.
 \item If $A \in \mathcal{T}_{\theta,F,x}(\mathcal{H})$ and $\theta(F_mV^*[A]^{1/2}U^*F_n^*x_n)= F_nU[A]^{1/2}VF_m^*x_m$, $ \forall n, m \in \mathbb{N}, \forall U,V \in \mathcal{B}(\mathcal{H})$, then $[A]^{1/2} \in \mathcal{S}_{\theta,F,x}(\mathcal{H})$ and $\tau_{\theta,F,x}(A)=\sigma_{\theta,F,x}([A]^{1/2})^2$.
 \end{enumerate}
  \end{proposition}
\begin{proof}
We have to argue only for ``=", others are proved in Lemma \ref{CHARACTERIZATIONST}.
\begin{enumerate}[\upshape(i)]
\item  $\tau_{\theta,F,x}(A)=\operatorname{Tr}_{\theta,F,x}([A])=\operatorname{Tr}_{\theta,F,x}(([A]^*[A])^{1/2})=\operatorname{Tr}_{\theta,F,x}([[A]])=\tau_{\theta,F,x}([A])$.
\item We write the proof of (i) in reverse way.
\item  Using  (v) in Theorem \ref{TRACEIPCONNECTION}, $\sigma_{\theta,F,x}([A]^{1/2})^2=\langle [A]^{1/2}, [A]^{1/2}\rangle =\operatorname{Tr}_{\theta,F,x}(([A]^{1/2})^*[A]^{1/2})=\tau_{\theta,F,x}(A)$.
\item We write the proof of (iii) in reverse way.
\end{enumerate}	
\end{proof}

\begin{lemma}\label{TAULEMMA}
Let $T \in \mathcal{B}(\mathcal{H})$, $A  \in \mathcal{T}_{\theta,F,x}(\mathcal{H})$ and $\theta(F_mV^*[A]^{1/2}U^*F_n^*x_n)= F_nU[A]^{1/2}VF_m^*x_m,$ $ \forall n, m \in \mathbb{N}, \forall U,V \in \mathcal{B}(\mathcal{H})$. Then $T[A]\in \mathcal{T}_{\theta,F,x}(\mathcal{H})$ and $|\operatorname{Tr}_{\theta,F,x}(T[A])|\leq \|T\|\tau_{\theta,F,x}(A)$.
\end{lemma}
\begin{proof}
Proposition \ref{PROPOSITIONTRACE} shows that $[A] \in \mathcal{T}_{\theta,F,x}(\mathcal{H})$ and (iv) in  Theorem \ref{TRACEIPCONNECTION} tells that $T[A]\in \mathcal{T}_{\theta,F,x}(\mathcal{H})$. Next,  using (iv) in Proposition \ref{PROPOSITIONTRACE}, $|\operatorname{Tr}_{\theta,F,x}(T[A])|=|\operatorname{Tr}_{\theta,F,x}(([A]^{1/2}T^*)^*[A]^{1/2})|=|\langle [A]^{1/2},[A]^{1/2}T^* \rangle |\leq \sigma_{\theta, F,x}([A]^{1/2})\sigma_{\theta, F,x}([A]^{1/2}T^*)\leq\sigma_{\theta, F,x}([A]^{1/2})^2\|T\| = \|T\|\tau_{\theta,F,x}(A)$.
\end{proof}
 \begin{theorem}\label{TRACETAUINEQUALITY}
 $A, B  \in \mathcal{T}_{\theta,F,x}(\mathcal{H})$, $T \in \mathcal{B}(\mathcal{H})$	and $\alpha \in \mathbb{K}$.
\begin{enumerate}[\upshape(i)]
\item $\tau_{\theta,F,x}(A^*)=\tau_{\theta,F,x}(A)$.
\item $\tau_{\theta,F,x}(\alpha A)=|\alpha|\tau_{\theta,F,x}(A)$.
\item If $A+B  \in \mathcal{T}_{\theta,F,x}(\mathcal{H})$, then $\tau_{\theta,F,x}(A+B)\leq\tau_{\theta,F,x}(A)+\tau_{\theta,F,x}(B)$.
\item $\tau_{\theta,F,x}(A)\geq0$.
\item Suppose $\tau_{\theta,F,x}(A)=0$. If there exist $a>0$ and $2\leq p <\infty$  such that
\begin{align*}
a\|h\|&\leq \left(\sum_{n=1}^{\infty}|\langle h, F_n^*x_n\rangle |^p\right)^\frac{1}{p},~ \forall h \in \mathcal{H}, 
\end{align*}		
then $A=0$. In particular, if $\{F_n^*x_n\}_n$ is a frame for $\mathcal{H}$, then  $A=0$.
\item $ \tau_{\theta,F,x}(TA)\leq \|T\|\tau_{\theta,F,x}(A)$, $ \tau_{\theta,F,x}(AT)\leq \|T\|\tau_{\theta,F,x}(A)$.
\item $|\operatorname{Tr}_{\theta,F,x}(A)|\leq\tau_{\theta,F,x}(A)$.
\item $\sigma_{\theta, F,x}(A)^2\leq\tau_{\theta,F,x}(A^*A).$
\item If $[A]\leq [B]$, then $\tau_{\theta,F,x}(A)\leq\tau_{\theta,F,x}(B)$.
\item If $\|A\|<1$, then $\tau_{\theta,F,x}(A^n)\rightarrow 0$ as $n\rightarrow \infty$.
\end{enumerate}		
 \end{theorem} 
  \begin{proof}
 Let $A=W[A]$ be the polar decomposition of $A$.  	
 \begin{enumerate}[\upshape(i)]
 \item Using Corollary \ref{TRACIALCOROLLARY},  $\tau_{\theta,F,x}(A^*)=\operatorname{Tr}_{\theta,F,x}([A^*])=\operatorname{Tr}_{\theta,F,x}(W[A]W^*)=\operatorname{Tr}_{\theta,F,x}(W^*W[A])=$ 
 
 $\sum_{n=1}^{\infty}\langle W^*W[A]F_n^*x_n, F_n^*x_n\rangle=\sum_{n=1}^{\infty}\langle W^*AF_n^*x_n, F_n^*x_n\rangle=\sum_{n=1}^{\infty}\langle [A]F_n^*x_n, F_n^*x_n\rangle=\operatorname{Tr}_{\theta,F,x}([A])=\tau_{\theta,F,x}(A)$.
 \item $\tau_{\theta,F,x}(\alpha A)=\operatorname{Tr}_{\theta,F,x}([\alpha A])=\operatorname{Tr}_{\theta,F,x}(|\alpha| [A])=|\alpha|\tau_{\theta,F,x}(A)$.
 \item   Let $B=W_1[B],A+B=W_2[A+B]$ be  polar decompositions of $B, A+B$, respectively. By using Lemma \ref{TAULEMMA}, $\tau_{\theta,F,x}(A+B)=\operatorname{Tr}_{\theta,F,x}([A+B])=\operatorname{Tr}_{\theta,F,x}(W_2^*(A+B))=\operatorname{Tr}_{\theta,F,x}(W_2^*W[A])+\operatorname{Tr}_{\theta,F,x}(W_2^*W_1[B])=|\operatorname{Tr}_{\theta,F,x}(W_2^*W[A])+\operatorname{Tr}_{\theta,F,x}(W_2^*W_1[B])| $ $\leq |\operatorname{Tr}_{\theta,F,x}(W_2^*W[A])|+|\operatorname{Tr}_{\theta,F,x}(W_2^*W_1[B])| $ $\leq\|W_2^*W\|\tau_{\theta,F,x}(A)+\|W_2^*W_1\|\tau_{\theta,F,x}(B)\leq\tau_{\theta,F,x}(A)+\tau_{\theta,F,x}(B)$.
 \item Since $[A]\geq0$, $\tau_{\theta,F,x}(A)=\sum_{n=1}^{\infty}\langle [A]F_n^*x_n, F_n^*x_n\rangle\geq0$. 
 \item $0=\tau_{\theta,F,x}(A)=\operatorname{Tr}_{\theta,F,x}([A])=\sum_{n=1}^{\infty}\langle [A]F_n^*x_n, F_n^*x_n\rangle=\sum_{n=1}^{\infty}\|[A]^{1/2}F_n^*x_n\|^2$  $\Rightarrow$ $[A]^{1/2}F_n^*x_n=0,\forall n \in \mathbb{N}$ $\Rightarrow$ $[A]F_n^*x_n=0,\forall n \in \mathbb{N}$ $\Rightarrow$ $a\|[A]h\|\leq (\sum_{n=1}^{\infty}|\langle [A]h, F_n^*x_n\rangle |^p)^{1/p}=(\sum_{n=1}^{\infty}|\langle h, [A]F_n^*x_n\rangle |^p)^{1/p}$ $=(\sum_{n=1}^{\infty}|\langle h, 0\rangle |^p)^{1/p}=0, \forall h \in \mathcal{H}$ $\Rightarrow$ $[A]=0$ $\Rightarrow$ $A=W[A]=0$.
 \item  We start from the polar decomposition of  $TA=W_3[TA]$ to get $ \tau_{\theta,F,x}(TA)=\operatorname{Tr}_{\theta,F,x}([TA])=\operatorname{Tr}_{\theta,F,x}(W_3^*TA)=\operatorname{Tr}_{\theta,F,x}(W_3^*TW[A])=\|W_3^*TW\|\tau_{\theta,F,x}(A)\leq \|T\|\tau_{\theta,F,x}(A)$. Similarly the polar decomposition of  $AT=W_4[AT]$ gives $ \tau_{\theta,F,x}(AT)=\operatorname{Tr}_{\theta,F,x}([AT])=\operatorname{Tr}_{\theta,F,x}(W_4^*AT)=\operatorname{Tr}_{\theta,F,x}(W_4^*W[A]T)$ $=\operatorname{Tr}_{\theta,F,x}(TW_4^*W[A])\leq\|TW_4^*W\|\operatorname{Tr}_{\theta,F,x}([A])\leq \|T\|\tau_{\theta,F,x}(A)$.
 \item  $|\operatorname{Tr}_{\theta,F,x}(A)|=|\operatorname{Tr}_{\theta,F,x}(W[A])|\leq\|W\|\tau_{\theta,F,x}(A) = \tau_{\theta,F,x}(A)$.
 \item From (iv) in Theorem \ref{TRACEIPCONNECTION}, $\sigma_{\theta, F,x}(A)^2=\operatorname{Tr}_{\theta,F,x}(A^*A)\leq\tau_{\theta,F,x}(A^*A)$.
 \item Using (ix) in Theorem \ref{TRACEIPCONNECTION}, $\tau_{\theta,F,x}(A)=\operatorname{Tr}_{\theta,F,x}([A])\leq \operatorname{Tr}_{\theta,F,x}([B])= \tau_{\theta,F,x}(B)$.
 \item Using (vi),  $0\leq \tau_{\theta,F,x}(A^n)\leq \|A^{n-1}\|\tau_{\theta,F,x}(A)\leq\|A\|^{n-1}\tau_{\theta,F,x}(A)\rightarrow 0$ as $n\rightarrow \infty$.
\end{enumerate}	 	
 \end{proof}
 \begin{corollary}
 Let 	$A \in \mathcal{T}_{\theta,F,x}(\mathcal{H})$ and  $\{G_n\}_n$ be an operator-valued orthonormal basis in  $\mathcal{B}(\mathcal{H},\mathcal{H}_0)$. Then the series $\sum_{n=1}^{\infty}|\langle AF_n^*x_n, G_n^*x_n\rangle|$ converges and $|\sum_{n=1}^{\infty}\langle AF_n^*x_n, G_n^*x_n\rangle|\leq\tau_{\theta,F,x}(A)$.
 \end{corollary}
\begin{proof}
From 	Theorem \ref{ORTHONORMALBASISCRITERION}, there exists a unique unitary  $U \in \mathcal{B}(\mathcal{H})$ such that $G_n=F_nU, \forall n \in \mathbb{N}$.	Now 	$UA \in \mathcal{T}_{\theta,F,x}(\mathcal{H})$ and hence $\sum_{n=1}^{\infty}|\langle UAF_n^*x_n, F_n^*x_n\rangle|$ converges i.e., $\sum_{n=1}^{\infty}|\langle AF_n^*x_n, U^*F_n^*x_n\rangle|=\sum_{n=1}^{\infty}|\langle AF_n^*x_n, G_n^*x_n\rangle|$ converges. A usage of  Theorem \ref{TRACETAUINEQUALITY} gives $|\sum_{n=1}^{\infty}\langle AF_n^*x_n, G_n^*x_n\rangle|=|\operatorname{Tr}_{\theta,F,x}(UA)|$ $\leq\|U\|\tau_{\theta,F,x}(A)=\tau_{\theta,F,x}(A)$.
\end{proof}

 \section*{Acknowledgements}
  The first author thanks the National
  Institute of Technology Karnataka (NITK), Surathkal for giving financial support. The third author is grateful to the Mohapatra Family Foundation for their support to pursue this research.
  
  \section*{Data Availability}
  The datasets generated during and/or analysed during the current study are available from the corresponding author on reasonable request.

 \bibliographystyle{plain}
 \bibliography{reference.bib}
 
 \end{document}